%% file: LevelsetTransport_Revision_a.tex
\documentclass[final]{siamltex}

\usepackage[usenames,dvipsnames,xcdraw,table]{xcolor}
\definecolor{LinkColor}{RGB}{0,0,255}               % used for links and so forth hyperref
\usepackage[top=3cm, bottom=3cm, left=1.5cm, right=3.5cm, heightrounded, marginparwidth=3.1cm, marginparsep=5mm]{geometry}                                          % spacing at boundary
\input{preamble.tex}
\usepackage{showlabels} 	                                    % show names of labels on margin (has to be loaded after preamble)
\graphicspath{{./images/}}

\begin{document}

%\title{A finite element extension method for the narrow band  level set equation }

\title{A narrow band finite element method for the  level set equation}

\author{ Maxim A. Olshanskii \thanks{Department of Mathematics, University of Houston, Houston, Texas 77204-3008, USA, (maolshanskiy@uh.edu), www.math.uh.edu/\string~molshan} \and Arnold Reusken \thanks{Institute for Geometry and Applied Mathematics, RWTH-Aachen University, D-52056 Aachen, Germany (reusken@igpm.rwth-aachen.de)} \and Paul Schwering \thanks{Institute for Geometry and Applied Mathematics, RWTH-Aachen University, D-52056 Aachen, Germany (schwering@igpm.rwth-aachen.de)}}

\maketitle

\begin{abstract}
A finite element method is introduced to track interface evolution governed by the level set equation. The method solves for the level set indicator function in a narrow band around the interface. An extension procedure, which  is essential for a narrow band level set method,  is introduced based on a finite element $L^2$- or $H^1$-projection combined with the ghost-penalty method. This procedure is formulated as a linear variational problem in a narrow band around the surface, making it computationally efficient and suitable for rigorous error analysis. The extension method is combined with a discontinuous Galerkin space discretization and a BDF time-stepping scheme. The paper analyzes the stability and accuracy of the extension procedure and evaluates the performance of the resulting narrow band finite element method for the level set equation through numerical experiments.
\end{abstract}

\section{Introduction}\label{sec_transport_introduction}
The level set method is a widely used numerical method for the representation and approximation of moving surfaces (or interfaces) \cite{osher2004level,Sethian1996}. The method is based on two central embeddings. The first one is the embedding of the surface as the zero level of a higher-dimensional scalar level set function, denoted by $\phi$, and the second one is the embedding, or extension, of the surface's velocity to a velocity field that transports this level set function.
For a closed hypersurface $\Gamma$: $[0,T] \to \mathbb{R}^d$ moving with speed $V_\Gamma$ in its normal direction, we consider the following \emph{Eulerian} formulation of its evolution in a domain $\Omega \subset \mathbb{R}^d$ such that the surface is contained in $\Omega$ for all times $t \in [0,T]$:
%\begin{align}
%	\frac{\partial \phi}{\partial t} + V_{\rm ex} |\nabla \phi| =0 \label{LS2},
%\end{align}
%where $|\cdot|$ denotes the Euclidean norm in $\mathbb{R}^d$ and $V_{\rm ex}$ is an extended speed such that $V_{\rm ex}=V_\Gamma$ on the zero level set of $\phi$. The equation is supplemented with the initial condition $\phi(\cdot, 0) = \phi_0(\cdot)$ such that $\phi_0(x)=0$ if and only if $x \in \Gamma(0)$. An alternative form of the level set equation \eqref{LS2} is
\begin{equation} \label{LS3}
	\dfrac{\partial \phi}{\partial t} + \bu \cdot \nabla \phi = 0,
\end{equation}
where $\bu:\,\Omega\to\R^d$ and the surface normal velocity are related by $\bu\cdot \bn=V_\Gamma$ on the zero level set of $\phi$. Here $\bn$ denotes the unit normal on $\Gamma(t)$.
The equation is supplemented with the initial condition $\phi(\cdot, 0) = \phi_0(\cdot)$ such that $\phi_0(x)=0$ if and only if $x \in \Gamma(0)$.

 In certain classes of applications, a canonical velocity field $\bu$ is given in $\Omega$ through the physics that drives the evolution of $\Gamma(t)$. A prototype example from this class is a two-phase immiscible flow problem where the evolution of the interface is driven by the continuous flow field of the two fluid phases \cite{Sussman2014,GrossReusken2011}. In other problem classes, the surface normal velocity is known \emph{only} at the surface and extension velocity $\bu$ has to be constructed.

Our interest in the topic of this paper stems from the modeling of fluid deformable surfaces \cite{hu2007continuum,jankuhn2018incompressible,torres2019modelling,nitschke2019hydrodynamic}. The governing partial differential equations in such models are posed on the evolving surface $\Gamma(t)$, and the surface geometry evolution may be part of the unknown solution. An Eulerian numerical framework for this and many other interface problems, such as geometric flows, commonly utilizes a level set method for the implicit representation of the surface, while the velocity needed in the level set equation results from the surface quantities.
In such a setting, the velocity field is available only on the surface or in a ``small neighborhood'' of it. For example, such a narrow band velocity may be provided by a TraceFEM for the spatial discretization of the surface PDE (cf. \cite{olshanskii2023eulerian}). Therefore, we are interested in a finite element solver for  \eqref{LS3} that is restricted to such a neighborhood of the surface, the so-called ``narrow band''.

In cases where a global velocity field $\bu$ is known, it may also be computationally efficient to solve the level set equation \eqref{LS3} in a narrow band instead of in the whole domain $\Omega$.
These locality requirements and efficiency advantages have led to the development of  narrow band level set methods, e.g.,  \cite{chopp1993computing,adalsteinsson1995fast,gomez2005reinitialization,cho2011direct,lee2014narrow, ngo2017efficient, xue2021new}.

In this paper, we introduce  a  new finite element narrow band level set method based on a specific extension technique that is outlined below.  For convenience, we assume that, in a narrow band,  a bulk velocity $\bu$ is available and consider the level set equation \eqref{LS3}. Narrow band methods for \eqref{LS3} require an extension procedure. To see why, assume that at time $t_n$ we are given an approximation $\phi_h^n$ of $\phi(\cdot,t_n)$ in a neighborhood $\Omega_\Gamma^n$ of $\Gamma^n_h\approx\Gamma(t_n)$ such that we a priori know that $\Gamma^{n+1}_h\subset\Omega_\Gamma^n$; cf. Fig.~\ref{fig_NarrowBand}. In the Eulerian setting, the numerical time integration over $[t_n,t_{n+1}]$ gives an approximation $\phi_h^{n+1}$ of $\phi(\cdot,t_{n+1})$ in $\Omega_\Gamma^n$, which defines $\Gamma^{n+1}_h$. To be able to perform the next time step $t_{n+1} \to t_{n+2}$, the computed approximated $\phi_h^{n+1}$ on $\Omega_\Gamma^n$ has to be extended to a suitable initial value on $\Omega_\Gamma^{n+1}$ such that $\Gamma^{n+2}_h\subset\Omega_\Gamma^{n+1}$, and so on.

In narrow band methods known from the literature, such an extension is typically based on a level set re-initialization technique.
A  well-known re-initialization technique is the fast marching method (FMM), which in the context of level set method approximates the distance from a point to a given surface. The fast marching method presented in the original paper \cite{Sethian1996b} is first-order accurate. Second-order accurate finite difference variants are presented in \cite{Sethian1999A,Chopp2001}. A drawback of using re-initialization in level set methods is the difficulty in controlling the change of the surface position. To address this issue several  finite difference-based approaches have been introduced \cite{Adalsteinsson1999,Chopp2001}.
Another class of re-initialization methods involves solving an evolutionary PDE to obtain a stationary solution that satisfies the Eikonal equation \cite{Sussman1999,Tornberg2000}. However, these methods have drawbacks due to parameter selection and high computational costs in a narrow band setting. An elliptic PDE-based approach \cite{basting2013minimization,xue2021new} avoids time marching to a stationary solution but requires solving a nonlinear problem and prescribing artificial boundary conditions in a narrow band.
To the best of our knowledge, all these methods  lack rigorous error estimates that would fit the finite element analysis of the entire narrow band discretization method.

In this paper, we present yet another approach for the extension of $\phi_h^{n+1}$ on $\Omega_\Gamma^n$ to a suitable initial value on $\Omega_\Gamma^{n+1}$. The extension method proposed here is based on a finite element $L^2$ or $H^1$ projection combined with the \emph{ghost-penalty method}. The ghost-penalty method is widely used to enhance the stability of finite element formulations~\cite{Burman2010a,burman2015cutfem}. More recently, it was also
suggested and analyzed as an implicit extension procedure for Eulerian finite element formulations of PDEs in time-dependent domains~\cite{LehrenfeldOlshanskii2019,von2022unfitted}. We utilize the ghost-penalty in this latter capacity. The resulting extension procedure can be formulated as a \emph{linear variational problem} in a narrow band around the surface. This makes it both computationally efficient and amenable to rigorous error analysis.

It is natural to combine this new extension  method, which is variational and finite element-based,   with a finite element
spatial discretization for the transport equation \eqref{LS3}. Discontinuous (DG) finite element methods  are known to be very suitable for convection-dominated or pure transport problems, cf. e.g., \cite{Brezzi2004,DolejsiFeistauer2015,DiPietroErn2012,FeistauerSvadlenka2004,burman2010interior}. We use   a DG scheme from \cite{Feistauer2016}.
In that paper a class of DG methods for the linear level set equation \eqref{LS3} is considered and for the space semidiscrete scheme with piecewise polynomials of degree $k$ an $L^2$ error bound $\mathcal{O}(h^{k+\frac12})$ is derived, provided the solution $\phi$ is sufficiently smooth. The analysis of that paper applies also if the velocity field is not divergence-free. Such a DG space discretization of the level set equation \eqref{LS3} can be combined with DG in time, cf. \cite{Feistauer2016}. We will use a  BDF method for the discrete time integration.

The main new contribution of this paper is the particular extension method that we propose. It can be combined with any reasonable finite element based discretization method for the level set equation.

One important aspect in any narrow band level set method is the choice of suitable boundary conditions for the level set equation on the inflow boundary part of the narrow domain. The extension method that we propose has some freedom in the choice of the domains  used for the local projection and for the ghost-penalty stabilization.  It turns out that for an appropriate choice of these domains one  can obtain higher order approximations of the zero level of $\phi$, i.e., of the surface, even if the numerical boundary data have only lower order accuracy.

In this paper we do \emph{not} use any re-initialization procedure in our narrow band level set method.  It turns out that even in cases with rather large deformations of $\Gamma$ we obtain reasonably accurate and stable interface recovery without the need for re-distancing $\phi_h$.
%Note that the normal vector field and curvatures of $\Gamma$ can be derived from a general level set function. By employing higher-order finite elements for $\phi_h$, this recovery becomes feasible locally for each simplex intersected by $\Gamma$, without relying on the values of $\phi_h$ in neighboring elements.
In applications with (very) large deformations  or in problems where an accurate approximation of the distance to $\Gamma_h$ is required, one may benefit from post-processing $\phi_h^n$ with some standard re-distancing technique (e.g., a variant of fast marching) to obtain $\psi_h^n \simeq \text{dist}(\cdot, \Gamma_h^n)$. Using this post-processing only every $k$th time step, with $k$ ``large'',  avoids systematic errors that could arise from using the re-distancing of $\phi_h^n$ in the extension procedure in every time step.

The remainder of this paper is organized as follows. In Section~\ref{sec_transport_overview} we specify our assumptions on the level set function and explain the structure of the narrow band discretization method. In Section~\ref{s:bc} we discuss the issue of numerical boundary conditions on the inflow boundary of the narrow band. In section \ref{sec_transport_basics} we explain the  methods we use for spatial and  time discretization of the level set equation. The main contribution of this paper is presented in Section~\ref{s:Ext}. We introduce the extension method and derive accuracy bounds for it. We also present here results of numerical experiments solely for the extension method. In Section~\ref{sec_transport_whole_algo} we put the different ingredients together and present the complete narrow band algorithm for discretization of the level set equation. Finally, in Section~\ref{sec_transport_experiments} we demonstrate the performance of  this algorithm in  numerical experiments.

\section{Level set equation and structure of the method} \label{sec_transport_overview}

We formulate the problem that we want to solve and outline the structure of our method.
 We assume that the evolving surface is given by the zero level of a sufficiently smooth level set function $\phi$,   $\Gamma(t)= \{\, x \in \Omega\,\mid\,\phi(x,t)=0\,\}$.  In  a neighborhood $\cO(\Gamma(t))= \Omega_\epsilon(t):=\{\, x \in \R^d~|~|\phi(x,t)| < \epsilon\,\}$ of the surface, with fixed $\epsilon >0 $,  we assume $\phi$ to be close to a signed distance function: For fixed constants $c_1$, $c_2$
\[
  0<c_1 \leq |\nabla \phi(x,t)| \leq c_2   \quad \text{for}~~x \in  \cO(\Gamma(t)), ~ t \in [0,T],
\]
holds.  Then the level set function is the unique solution of
\begin{equation*}
	\dfrac{\partial \phi}{\partial t}  + \bu \cdot \nabla \phi   = 0 \quad \text{on}~ \bigcup_{t \in (0,T]}\cO(\Gamma(t))\times\{t\},
\end{equation*}
with a suitable initial condition $\phi(x,0) = \phi_0(x)$ for $x \in  \cO(\Gamma(0))$.
%\todo{Conditions on $\bu$? What if $\bu$ has a very strong tangential component?}
Note that we do not need inflow boundary conditions because the spatial boundary of the space-time neighborhood $\bigcup_{t \in (0,T]}\cO(\Gamma(t))\times \{t\}$ is a characteristic boundary. \emph{Our primary interest is in accurately recovering} $\Gamma(t)$ rather than approximating $\phi$ over the entire neighborhood $\bigcup_{t \in [0,T]} \cO(\Gamma(t)) \times \{t\}$. Therefore, it is sufficient to have an accurate approximation of $\phi$ on a smaller subdomain that still contains $\Gamma(t)$ for $t \in [0,T]$. This smaller domain  will be specified below.

We first outline the time stepping procedure. We aim to use narrow bands $\Omega_\G^n$ with a fixed width and a varying time step size, such that $\G^{n+1}$ is contained in $\Omega_\G^n$. The choice of the time step sizes $\Delta t_n$ will be discussed later. Consider time nodes $t_n$, $n = 1, \dots, N$, with $0=t_0 < t_1 < \dots < t_N = T$.  We use the notations $f^n(\cdot) = f(\cdot, t_n)$ and  $\G^n = \G(t_n)$. In   the background domain  $\Omega$ we assume a family of  shape regular quasi-uniform simplicial triangulations $\{\cT_h\}_{h >0}$. The subset of simplices that intersect $\Gamma(t_n)$ is denoted by $\cT_\G^n=\{\, T \in \cT_h~|~{\rm meas}_{d-1}(T \cap \G^n) >0\,\}$. To define the narrow bands used in the method, we introduce the notation for a set of neighboring elements of a sub-triangulation $\omega_h \subset \cT_h$:
\begin{equation*}
    \cN^1(\omega_h)  := \left\{T \in \cT_h ~|~ T \cap \omega_h \neq \emptyset \right\} \quad \text{and} \quad \cN^j(\omega_h) = \cN^1(\cN^{j-1}(\omega_h)), \quad j \geq 2.
\end{equation*}

We define the narrow band $\Omega_\G^n$, $1 \leq n \leq N$, by adding a fixed number $J$ of layers of elements to the ``cut'' elements $\cT_\G^n$:
\begin{equation} \label{narrowb}
    \Omega_\G^n = \cN^J(\cT_\G^n).
\end{equation}
Hence, for $h$ sufficiently small we have embedding $\cT_\G^n \subset \Omega_\Gamma^n \subset \cO(\Gamma(t_n))$. We chose the time step  size $\Delta t_n$ sufficiently small, such that the surface $\G^{n+1}$ is contained in $\Omega_\Gamma^n$, $n=0,1,\ldots,N-1$, cf. Fig.~\ref{fig_NarrowBand} for a 1D  illustration. In Section~\ref{sec_transport_whole_algo} we discuss how this condition is handled in our algorithm.

\begin{figure}[!ht]
    \centering
    \includegraphics[width=0.5\textwidth]{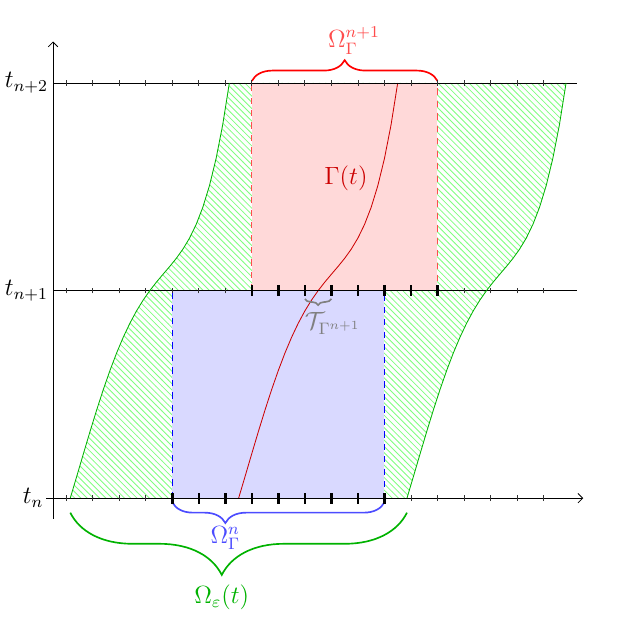}
    \caption{Sketch of narrow band $\mathcal{O}(\Gamma(t))=\Omega_\epsilon(t)$ and successive time slabs $\Omega_\Gamma^n \times [t_n,t_{n+1}]$, $\Omega_\Gamma^{n+1} \times [t_{n+1},t_{n+2}]$}
    \label{fig_NarrowBand}
\end{figure}

In an Eulerian framework, we (approximately) solve the level set equation   on a space-time subdomain $\Omega_\Gamma^n \times [t_n,t_{n+1}]$ and  repeat this on the next space-time subdomain $\Omega_\Gamma^{n+1} \times [t_{n+1},t_{n+2}]$. Clearly, for a well-posed formulation of the level set equation on $\Omega_\Gamma^n \times [t_n,t_{n+1}]$ we need initial data on $\Omega_\Gamma^n$ for $t=t_n$ and boundary data on the inflow part of the piecewise planar boundary of $\Omega_\Gamma^n$.
The inflow boundary is given by
\[ \partial\Omega_{\Gamma,in}^n(t) =\{\, x \in \partial\Omega_{\Gamma}^n~|~ \bu(x,t) \cdot \bn_{\Omega_{\G^n}}(x) <0\, \}, \quad t \in [t_n,t_{n+1}],
\]
where $\bn_{\Omega_{\G^n}}$ denotes the outward pointing unit normal on $\partial\Omega_{\Gamma}^n$. For simplicity, in the remainder we assume that within each time interval the inflow boundary is constant, i.e., $\partial\Omega_{\Gamma,in}^n(t)$ is independent of $t \in [t_n,t_{n+1}]$. This inflow boundary is denoted by $\partial\Omega_{\Gamma,in}^n$. The Dirichlet boundary data on this inflow boundary are denoted by $\phi_D^n$.

Given the narrow bands $\Omega_\Gamma^n$, $n=0,\ldots, N-1$, that satisfy the conditions specified above, the Eulerian narrow band method that we treat in this paper has the structure outlined in Algorithm~\ref{alg}.

\begin{algorithm}[h!] \caption{One time step structure}\label{alg}
\begin{itemize}
		\item[a)] Given $\phi_h^n$, defined on $\Omega_\Gamma^n$,  specify boundary conditions $\phi_D^n$ on the inflow boundary $\partial\Omega_{\Gamma,in}^n$.
		\item[b)]  Given the initial condition $\phi_h^n$ and boundary data $\phi_D^n$, solve the level set equation approximately on $\Omega_\Gamma^n \times [t_n,t_{n+1}]$. The result, defined on $\Omega_\Gamma^n$, is denoted by $\tilde \phi_h^{n+1} = S(\phi_h^n,\phi_D^n)$.
		\item[c)] Find $\phi_h^{n+1}$ as an extension of $\tilde \phi_h^{n+1}$ from $\Omega_\Gamma^n$  to $\Omega_\Gamma^{n+1}$.
	\end{itemize}
\end{algorithm}

An approach for  step a) is explained in Section~\ref{s:bc}. For solving the level set equation in step b) we apply a Galerkin DG discretization in  space combined with BDF in time; see Section~\ref{sec_transport_basics}. The main new contribution of this paper is the extension method used in step c), which is explained  in Section~\ref{s:Ext}.
%In the sections below we explain these methods.

\section{Construction of inflow boundary data} \label{s:bc}
The   Dirichlet boundary data on $\partial\Omega_{\Gamma,in}^n$
will be given by a (linear) mapping $B_n: \phi_h^n \to \phi_D^n$, with $\phi_D^n=\phi_D^n(x,t)$, $x \in \partial\Omega_{\Gamma,in}^n$, $t \in [t_n,t_{n+1}]$. To define $B_n$, we use a straightforward approach.

One obvious possibility is to take the  boundary values at time $t_n$,
\begin{equation} \label{bnd_data_BDF1_lo}
   \phi_D^n= B_n\phi_h^n:=(\phi_h^n)_{|\partial\Omega_{\Gamma,in}^n},
\end{equation}
i.e., we simply take the values of the approximation at $t=t_n$ restricted to the inflow boundary. These boundary data are independent of $t \in [t_n,t_{n+1}]$.  If $\phi_h^n(x) = \phi(x,t_n)$ for $x \in \Omega_\Gamma^n$, this construction yields boundary data that are  first order accurate in $\Delta t$. A second order (in time) approximation is obtained using
\begin{equation} \label{bnd_data_BDF1_ho}
  \phi_D^n = B_n(\phi_h^n, \bu^n) := (\phi_h^n)\big|_{\partial\Omega_{\Gamma,in}^n} - (t-t_n)\big(\bu^n \cdot \nabla \phi_h^n\big)\big|_{\partial\Omega_{\Gamma,in}^n}.
\end{equation}
Both options \eqref{bnd_data_BDF1_lo} and \eqref{bnd_data_BDF1_ho} are natural choices in a BDF1 time stepping procedure. If one uses BDF2 it is natural to incorporate information of $\phi_h^{n-1}$ in the boundary data. Analogously to \eqref{bnd_data_BDF1_lo} we can use
\begin{equation} \label{bnd_data_BDF2_lo}
    \phi_D^n = B_n(\phi_h^n, \phi_h^{n-1}) := 2(\phi_h^n)_{|\partial\Omega_{\Gamma,in}^n} - (\phi_h^{n-1})_{|\partial\Omega_{\Gamma,in}^n},
 \end{equation}
which is second order accurate in $\Delta t$, or a higher order approximation along the same lines as \eqref{bnd_data_BDF1_ho}. We will discuss the choice of the boundary data in Section~\ref{sec_transport_experiments}.

The default boundary data operator that we use in our numerical experiments is the one in \eqref{bnd_data_BDF2_lo} for BDF2. We shall discuss below in Section~\ref{sec_transport_whole_algo} that one can obtain higher order discretization accuracy close to the surface even if low order (in $\Delta t$) accurate Dirichlet boundary data are used. This effect is essentially caused by  the transport nature of the  level set  problem, which implies that errors arising on a level set $|\phi(\cdot,t_n)|=ch$ do not enter the domain formed by the level sets $|\phi(\cdot,t)| < ch$ for $t \geq t_n$.

\section{Discretization of the level set equation} \label{sec_transport_basics}
This section explains the space and time discretization that we use on one time slab $\Omega_\Gamma^n \times [t_n,t_{n+1}]$. We assume given Dirichlet boundary data $\phi_D^n$ on the inflow boundary $\partial\Omega_{\Gamma,in}^n$. To simplify the notation, we delete the time slab superscript $n$ and write $\Omega_\Gamma$, $\phi_h$, $\phi_D$ instead of $\Omega_\Gamma^n$, $\phi_h^n$, $\phi_D^n$ in this section.
We employ the discontinuous Galerkin scheme detailed in \cite{Feistauer2016}, paired with a BDF-scheme for the time derivative. Unlike most of the methods in the literature, this DG method does not require the velocity $\bu$ to be divergence-free, which is beneficial for  applications that motivated this work. %See Remark \ref{remark_u_not_divfree} for a discussion of this topic.

We first consider the space discretization.  The finite element space is given by
\begin{equation}\label{eqn_def_DGSpace_transport}
    V_h^{\text{DG}}(\Omega_\Gamma) := \{\, \psi \in L_2(\Omega_\Gamma) ~|~ \psi|_{T} \in P_{k}(T) ~ \forall T \in (\cT_h \cap \Omega_\Gamma)\, \}.
\end{equation}
As is usual in DG type methods, we need the upwind fluxes and a suitable jump operator across the faces in  $\cT_h$. We split the boundary of an element $T \in (\cT_h \cap \Omega_\Gamma)$ into its inflow and outflow part
\begin{align*}
    \partial T^{-}(t) &:= \left\{\, x \in \partial T ~:~ \bu(x, t) \cdot \bn_T(x) < 0 \right\}, \\[1ex]
    \partial T^{+}(t) &:= \left\{\, x \in \partial T ~:~ \bu(x, t) \cdot \bn_T(x) \geq 0 \right\},
\end{align*}
where $\bn_T$ denotes the outward pointing unit normal on $T$, cf. Fig.~\ref{fig_common_face_triangles}.

% Due to the discontinuity of $\phi_h$ over the faces of the tetrahedra, the values of the finite element functions in the boundary integral are not uniquely defined. More precisely, let \tT be an adjacent tetrahedron to $T$, such that $T$ and \tT are connected over the face $F$, see figure \ref{fig_common_face_triangles} for a visualization. Thus, $\bn_T$ denotes the normal to $F$ pointing outward of $T$ (and into \tT). For the evaluation of the finite element function \phih on $F$, one can choose $\phih|_{F} = \phih|_{T}$ or $\phih|_{F} = \phih|_{\tT}$.

\begin{figure}[ht!]
    \centering
    \begin{tikzpicture}[scale=0.8]
        % Define common vertices
        \coordinate (A) at (0,2);
        \coordinate (B) at (0,-2);
        \coordinate (C) at (3.464, 0);  % Adjust the height to control triangle size
        \coordinate (D) at (-3.464, 0); % Adjust the height to control triangle size

        % Triangle T
        \draw (A) -- (B) -- (C) -- cycle;

        % Triangle \tilde{T}
        \draw (A) -- (D) -- (B);

        \node at (0.3,1.3) {$F$};
        \node at (1.5,0) {$T$};
        \node at (-1.8,0) {$\tilde{T}$};

        % normal n
        \draw[->,thick] (0,0) -- ++(-1.4, 0);
        \node at (-0.7,0.3) {$\bn_T$};

    \end{tikzpicture}
    \caption{Adjacent Triangles $T$ and $\tilde{T}$ sharing the face $F$.}
    \label{fig_common_face_triangles}
\end{figure}
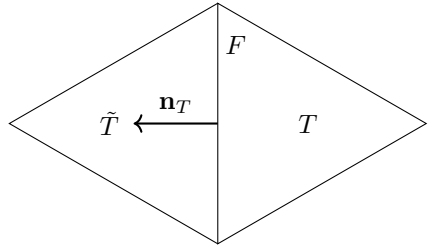

We introduce the jump operator $[\psi_h]=[\psi_h]|_{\partial T}=\psi_h|_{T}-\psi_h|_{\tT}$, $\psi_h \in V_h^{\text{DG}}$. For a derivation of the following semi-discrete DG finite element method we refer to \cite{Feistauer2016}: Find $\phi_h(t) \in V_h^{\text{DG}}(\Omega_\Gamma)$ such that $\phi_h(t_n)=\phi_h^n$ and for $t \in [t_n,t_{n+1}]$:
\begin{multline}\label{eqn_semi_discrete_levelset_transport}
 \mathlarger{\mathlarger{\sum}}_{T \in (\cT_h \cap \Omega_\Gamma)} \left( \int_T \frac{\partial \phi_h}{\partial t}  ~ \psi_h \dx + \int_T \left(\bu \cdot \nabla\phih \right) \psi_h \dx - \hspace*{-0.2cm} \int_{\partial T^- \backslash \partial \Omega_\Gamma} \hspace{-0.2cm} [\phih] \psi_h \left(\bu \cdot \bn_T \right) \ds  \right) \\[1ex]
        - \int_{\partial \Omega_{\Gamma,in}} \phih \psih \left(\bu \cdot \bn_{\Omega_\Gamma}\right) \ds = - \int_{\partial\Omega_{\Gamma,in}} \phi_D \psi_h \left( \bu \cdot \bn_{\Omega_\Gamma}\right) \ds
        \quad \text{for all } \psi_h \in V_h^{\text{DG}}(\Omega_\Gamma).
\end{multline}
Note that $\bu$ and $\phi_D$ may depend on $t$.
For this  formulation on a  fixed polygonal domain $\Omega \subset \R^d$ (instead of our $h$ dependent domain $\Omega_\Gamma$), $L^2$-norm error bounds of order $h^{k+\frac12}$ are derived in  \cite[Theorem 3.6,3.7]{Feistauer2016}.

For the discretization in time, different standard approaches are known in the literature. In \cite{Feistauer2016}, a discontinuous Galerkin scheme in time is used. In \cite{MarchandiseRemacleChevaugeon2006, DiPietroLoForteParolini2006}, a DG method in space is combined with a Runge-Kutta time discretization scheme. We opt for low-order BDF schemes for the following reasons: they necessitate only one evaluation of $\bu$ per time step and are easily implemented within our finite element software. Specifically, the BDF1 scheme (implicit Euler) leads to the following fully discrete problem within a single time interval: For given $\phih^n \in V_h^{\text{DG}}(\Omega_\Gamma^n)$, determine $\tilde \phi_h^{n+1} \in V_h^{\text{DG}}(\Omega_\Gamma^n)$ such that
\begin{equation}\label{eqn_discrete_levelset_transport}
    \begin{split2}
        && \mathlarger{\mathlarger{\sum}}_{T \in (\cT_h \cap \Omega_\Gamma^n)} \left( \int_T \dfrac{\tilde \phi_h^{n+1} - \phih^{n}}{\Delta t_n} \, \psi_h \dx + \int_T \left(\bu^{n+1} \cdot \nabla \tilde \phi_h^{n+1} \right) \psi_h \dx - \hspace*{-0.2cm} \int_{\partial T^- \backslash \partial \Omega_\Gamma^n} \hspace{-0.2cm} [\tilde \phi_h^{n+1}] \psi_h \left(\bu^{n+1} \cdot \bn_T\right) \ds  \right) \\[1ex]
        &-& \int_{\partial \Omega_{\Gamma,in}^n} \tilde \phi_h^{n+1} \psih \left(\bu^{n+1} \cdot \bn_{\Omega_\Gamma^n}\right) \ds = - \int_{\partial\Omega_{\Gamma,in}^n} \phi_D^{n+1} \, \psi_h \left(\bu^{n+1} \cdot \bn_{\Omega_\Gamma^n}\right) \ds
        \qquad \text{for all } \psi_h \in V_h^{\text{DG}}(\Omega_\Gamma^n).
    \end{split2}
\end{equation}

The extension of this method to higher order BDF (used in the numerical experiments in Section~\ref{sec_transport_experiments}) is straightforward.

\begin{remark} \rm
    The specific choice of discretization method for the level set equation in step b) of Algorithm~\ref{alg} is not crucial for the performance of our narrow band level set method. However, for the extension method discussed in the next section, it is essential to use finite element spaces. Therefore, using a finite element discretization for the level set equation is natural. Given that the level set equation is a pure transport equation, a finite element discontinuous Galerkin (DG) technique is an appropriate choice. Related DG techniques can be found in \cite{SudirhamVanDerVegtVanDamme2006,SudirhamVanDerVegt2008,CockburnShu1989, CockburnHouShu1990}. Alternatively, a stabilized conforming finite element method, such as the streamline diffusion finite element method (cf. \cite{RoosStynes}), could be used.
\end{remark}

In the DG method, we obtain finite element approximations that in general are discontinuous across the edges/faces  of the triangulation. One variant of the extension method treated below applies to continuous $H^1$ conforming finite element functions. We therefore use a simple quasi-interpolation operator that maps (with optimal order of accuracy) the DG finite element functions into the $H^1$ conforming Lagrange finite element space with the same polynomial degree. We apply a standard Oswald interpolation in which different values of a DG finite element function around a finite element node are averaged to determine the value of the continuous finite element function at that node, cf. \cite{Guermond2017}. Using this quasi-interpolation operator in the $H^1$ conforming finite element space we obtain a continuous approximation $\Gamma_h(t_n)$ of the zero level $\Gamma(t_n)$.

%\todo{PS: Such interpolation is not necessary for the $L^2$-projection based extension procedure, but it's important for the calculation of the zero level. Should we comment on this?}

\section{Extension of the level set function} \label{s:Ext}

In this section we propose a particular extension method for the step~c) of Algorithm~\ref{alg}, the extension of $\tilde\phi_h^{n+1}$ from $\Omega_\Gamma^n$ to $\Omega_\Gamma^{n+1}$. The proposed extension method uses a \emph{ghost-penalty technique}. Several variants of the ghost-penalty are known in the literature. The original version, known as local projection stabilization (LPS), was introduced in \cite{Burman2010a}. Another variant is the normal derivative jump stabilization, extensively discussed in works such as \cite{BurmanHansbo2012,  MassingLarsonLoggRognes2014} for its first-order version, and in \cite{Lehrenfeld2017, SchottWall2014} for higher-order versions.

%This technique has been used in the literature for different purposes, e.g., stabilization of fictituous domain methods and Nitsche type methods or extension of functions in the setting of CutFEM \cite{addliterature} \todo{add references}.
In our method we use the ``volumetric jump'' or ``direct'' formulation of the ghost penalty stabilization, derived in \cite{Preuss2018} and further investigated in \cite{LehrenfeldOlshanskii2019, von2022unfitted}. For a comparison of different ghost penalty techniques, readers are referred to \cite{LehrenfeldOlshanskii2019}. In this work, we restrict ourselves to the direct formulation. To the best of our knowledge, there is no existing literature where the ghost penalty technique was applied within the framework of a narrow band level set method.

We derive and analyze the extension  method in a more general setting where a function $\tilde \phi \in H^1(\Omega_h^p)$ is extended from a projection-subdomain $\Omega_h^p\subset\Omega$ to some extension-domain $\Omega_h^e$ such that $\Omega_h^p\subset\Omega_h^e\subset\Omega$. The result of such an extension will be a continuous finite element function $\phi_h$. There is no requirement for $\tilde \phi$ and $\phi_h$ to coincide in $\Omega_h^p$. We then apply this method for an extension from $\Omega_\Gamma^n$ to $\Omega_\Gamma^{n+1}$ in step c) of Algorithm~\ref{alg}, cf. Section~\ref{sec_transport_whole_algo}.

We assume that both $\Omega_h^p$ and $\Omega_h^e$ are unions of simplices from $\cT_h$. For simplicity, we identify a set of simplices with the domain defined by the union of the elements. We need some further assumption concerning the geometry of  $\Omega_h^p$ and $\Omega_h^e$. Essentially we want $\Omega_h^e$ to be an extension of the domain $\Omega_h^p$ by a few additional layers of elements. In  our application, both   $\Omega_h^p$ and  $\Omega_h^e$ are tubular neighborhoods of width $\sim h$ around a given surface, cf.  Fig.~\ref{fig_GP_elements_faces}. This specific  structure of $\Omega_h^p$ and $\Omega_h^e$, however, is not needed in the analysis presented in this section. Instead, the following more general assumption is made.

\begin{assumption} \label{Ass1} \rm
    We assume that $\Omega_h^p$ consists of a subset of elements $T\in \cT_h$ such that $\Omega_h^p$ is path-connected  and such that any $T \subset \Omega_h^p$ has a common face (3D) or edge (2D) with another $T \subset \Omega_h^p$. The union of simplices $\Omega_h^e$ is such that it contains $\Omega_h^p$ and for any $T \in \Omega_h^e \backslash \Omega_h^p$ there is a path of length at most $ch$, with a uniform constant $c$, that is completely contained in $\Omega_h^e$ and connects $T$ with an element contained in $\Omega_h^p$.
\end{assumption}

The space of \emph{continuous} finite element functions on a triangulated domain $\omega_h$ is denoted by
\begin{equation*}
    V_h(\omega_h) := \{\psi_h \in C(\omega_h): \psi_h|_T \in P_{k}(T) ~\forall T \in \omega_h\}.
\end{equation*}

 In the definition of the ghost-penalty stabilization below  we need the difference between $\Omega_h^p$ and its extension, i.e., $\Omega_h^\text{diff} := \Omega_h^e \backslash \Omega_h^p$.  We define the set of ghost penalty faces $\cF_h^{\text{GP}}$ by
\begin{equation*}
    \cF_h^{\text{GP}} := \left\{F \subset \partial T ~|~ T\in \Omega_h^\text{diff}, ~ F \not \subset \partial \Omega_h^e \right\} \cup \left\{F \subset \partial T~|~ T \in \Omega_h^p, ~ T\cap \partial \Omega_h^p \neq \emptyset\right\}.
    % \cF_h^{\text{GP}} := & \left\{F = \overline{T}_1 \cap \overline{T}_2 ~|~ T_1,\, T_2 \in \Omega_h^\text{diff} \cup \left(\cN(\Omega_h^\text{diff}) \cap \Omega_h^p\right),\, \overline{T}_1 \neq \overline{T}_2 ,\, \mathcal{H}_{d-1}(F)>0 \right\}, & \cup \left\{F = \overline{T}_1 \cap \overline{T}_2 ~|~ T_1 \in \cN(\Omega_h^\text{diff}) \cap \Omega_h^p, \, T_2 \notin \cN(\Omega_h^\text{diff}) \cap \Omega_h^p, \mathcal{H}_{d-1}(F)>0 \right\}  ,
\end{equation*}

This definition ensures that all faces inside the ``new'' elements $\Omega_h^\text{diff}$ and their ``inner neighbors'' in $\Omega_h^p$ are included. The faces on $\partial \Omega_h^e$ are not included, since these do not have two neighboring elements in $\Omega_h^e$. See Figure \ref{fig_GP_elements_faces} for an illustration of $\Omega_h^p$, $\Omega_h^e$, and the corresponding faces $\cF_h^{\text{GP}}$.
\begin{figure}[ht!]
    \centering
    \includegraphics[width=0.7\textwidth]{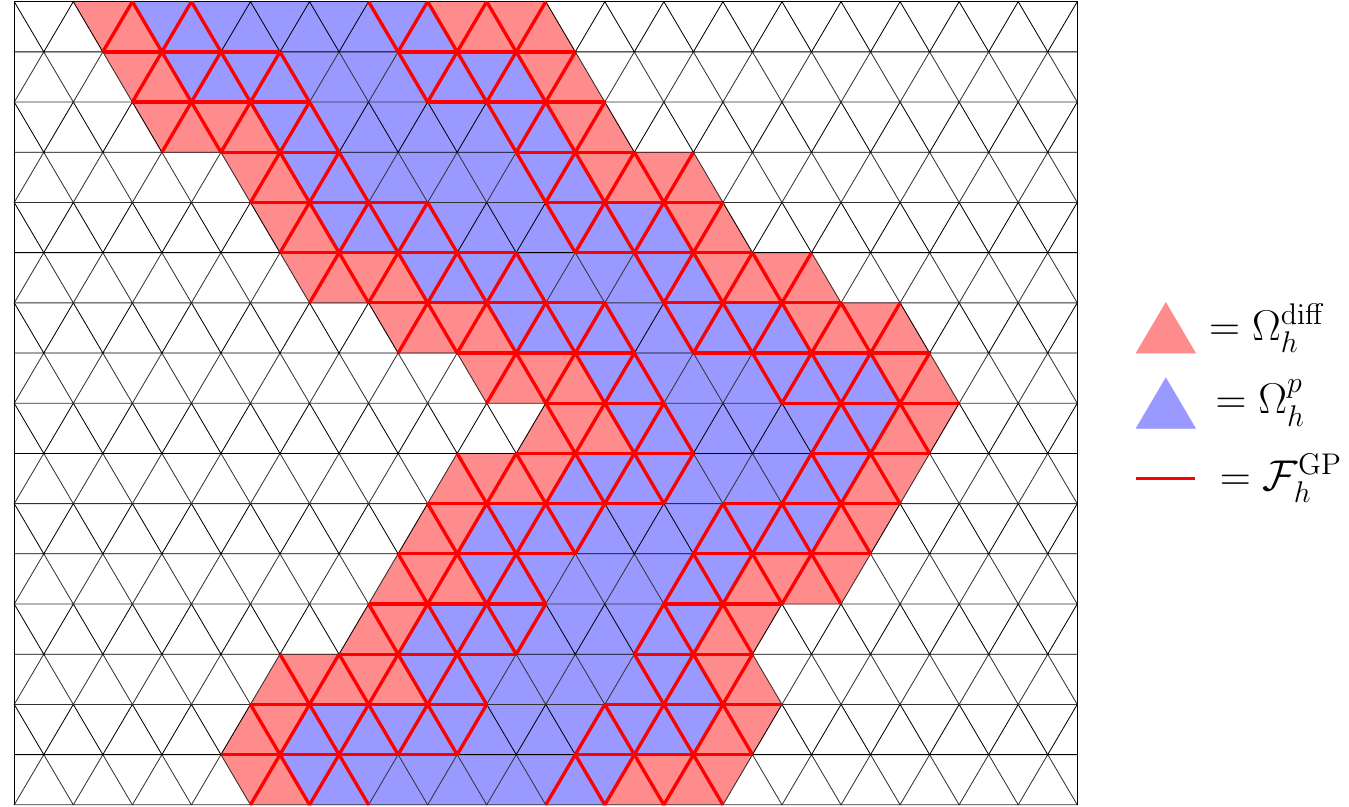}
    \caption{$\Omega_h^p$, $\Omega_h^e$, and the corresponding faces $\cF_h^{\text{GP}}$.}
    \label{fig_GP_elements_faces}
\end{figure}

For a face $F = {T}_1 \cap {T}_2 \in \cF_h^{\text{GP}}$ its corresponding patch is defined as $\omega(F) = {T}_1 \cup {T}_2$. For a  finite element function $\psi \in V_h(\omega(F))$, we use the notation $\psi_i = \cE^P \psi|_{T_i}$ with the canonical polynomial extension operator $\cE^P : P_{k}(T) \rightarrow P_{k}(\R^d)$. The ghost penalty bilinear form is defined as
\begin{equation}\label{eqn_def_GPExt_operator}
    s_h^{(\alpha)}(\phi, \psi) := \gamma^\text{ext} h^{-\alpha}  \sum_{F \in \cF_h^{\text{GP}}} \int_{\omega(F)} (\phi_1 - \phi_2) (\psi_1 - \psi_2 ) \dx, \quad \text{for } \phi, \psi \in V_h(\Omega_h^e),
\end{equation}
with parameters $\gamma^\text{ext} > 0$ and $\alpha$.

 For the error analysis, we need the ghost penalty bilinear form $s_h^{(\alpha)}(\cdot,\cdot)$ to be well defined not only for finite element functions. We define $\psi_i = \cE^P \Pi_{T_i} \psi|_{T_i}$ for $i=1,2$, where $\Pi_{T_i}$ is the $L^2(T_i)$-projection into $P_{k}(T_i)$, $i=1,2$, cf. \cite{LehrenfeldOlshanskii2019}.  If $\psi \in V_h(\Omega_h^e)$ it holds $\Pi_{T_i} \psi|_{T_i} = \psi|_{T_i}$. Using this,  the bilinear form \eqref{eqn_def_GPExt_operator} is extended to $L^2(\Omega_h^e) \times L^2(\Omega_h^e) $.

 We  %consider both  the $L^2$- and the $H^1$-projection and
 now define the following bilinear forms:
\begin{equation*}
    a_h^{\text{ext}}(\phih, \psih) := (\phih, \psih)_{\Omega_h^p} + s_h^{(0)}(\phih, \psi_h) \quad \text{and} \quad a_{1,h}^{\text{ext}}(\phih, \psih) := (\phih, \psih)_{1, \Omega_h^p} + s_h^{(2)}(\phih, \psi_h),
\end{equation*}
with $(\cdot, \cdot)_{\Omega_h^p}$ and $(\cdot, \cdot)_{1, \Omega_h^p}$, the $L^2$- and $H^1$-inner product on $\Omega_h^p$, respectively. The scaling with $\alpha=0$ and $\alpha=2$ in these bilinear forms is based on results from the literature.
%\cite{addreferences} \textbf{AR: discuss with Paul. MO: for $H^1$ the scaling is analyzed in \cite{LehrenfeldOlshanskii2019} but I failed to find analysis for $L^2$ case} .

The corresponding extension problems read:  For a given function $\tilde \phi$ defined on $\Omega_h^p$, determine $\phih \in V_h(\Omega_h^e)$ such that
\begin{subequations}\label{eqn_def_GPExt_problem}
    \begin{math}
        a_h^\text{ext}(\phih, \psih) &=& (\tilde \phi, \psi_h)_{\Omega_h^p} \qquad &\text{for all}& \quad  \psih \in V_h(\Omega_h^e), \label{eqn_def_GPExt_problem_L2} \\[1ex]
        a_{1,h}^\text{ext}(\phih, \psih) &=& (\tilde \phi, \psi_h)_{1,\Omega_h^p} \qquad &\text{for all}& \quad \psih \in V_h(\Omega_h^e). \label{eqn_def_GPExt_problem_H1}
    \end{math}
\end{subequations}
In section~\ref{sec_transport_whole_algo} we use this extension method in the setting of the narrow band level set method, cf. component c) of the Algorithm~\ref{alg}. We analyze both formulations in the next section.

\subsection{An error analysis of the extension.}
%In this section, we derive basic properties of the extension methods \eqref{eqn_def_GPExt_problem}.
We derive error estimates for the extension methods in the $L^2$- and $H^1$-norm on $\Omega_h^e$ denoted by $\|\cdot\|_{\Omega_h^e}$ and $\|\cdot\|_{1, \Omega_h^e}$, respectively. It is also helpful to introduce notations for the norms $\|\cdot\|_a = \sqrt{a_h^\text{ext}(\cdot, \cdot)}$ and $\| \cdot \|_{1,a} = \sqrt{a_{1,h}^\text{ext}(\cdot, \cdot)}$.

We summarize some useful results from \cite{LehrenfeldOlshanskii2019} in the following lemma.
\begin{lemma} \label{LemmaGP}
    Let $I_h$ be the Lagrange interpolation operator into the finite element space $V_h$.  For arbitrary $\psih \in V_h\left(\Omega_h^e\right)$ and $\psi \in H^{k+1}\left(\Omega_h^e\right)$ there holds
    \begin{align}
        \|\psih\|_{\Omega_h^e}^2 &\lesssim  \|\psih\|_{\Omega_h^p}^2 + s_h^{(0)}(\psih, \psih), \label{eqn_from_LO19_GP_L2_estimate}\\[1ex]
        \|\nabla \psih\|_{\Omega_h^e}^2 &\lesssim  \|\nabla \psih\|_{\Omega_h^p}^2 + s_h^{(2)}(\psih, \psih), \label{eqn_from_LO19_GP_h1_estimate}\\[1ex]
        s_h^{(0)}(\psi, \psi) &\lesssim  h^{2k+2} \|\psi\|_{H^{k+1}(\Omega_h^e)}^2, \label{eqn_from_LO19_GP_continuity_estimate}\\[1ex]
        s_h^{(0)}(\psi-I_h \psi, \psi-I_h \psi) &\lesssim  h^{2k+2}\|\psi\|_{H^{k+1}(\Omega_h^e)}^2. \label{eqn_from_LO19_GP_interpolation_estimate}
    \end{align}
Here and below the notation $a\lesssim b$ means that the inequality $a\le c\,b$ holds with a constant $c$ which is independent of $h$ and the position of $\Gamma$ in the background mesh.
\end{lemma}
\begin{proof}
    See \cite[Lemma 5.2]{LehrenfeldOlshanskii2019} for the proof of inequalities \eqref{eqn_from_LO19_GP_L2_estimate}--\eqref{eqn_from_LO19_GP_h1_estimate} and \cite[Lemma 5.5]{LehrenfeldOlshanskii2019} for the proof of inequalities \eqref{eqn_from_LO19_GP_continuity_estimate}--\eqref{eqn_from_LO19_GP_interpolation_estimate}. % Following Remark 5.4 from \cite{LehrenfeldOlshanskii2019}, we can assume that assumption 5.3 is fulfilled.
\end{proof}

From \eqref{eqn_from_LO19_GP_L2_estimate} and \eqref{eqn_from_LO19_GP_h1_estimate} it follows that
\begin{math}
    a_h^\text{ext}(\psih, \psih)  &\gtrsim& \| \psih \|_{\Omega_h^e}^2 \quad &\text{for all }& \psih \in V_h(\Omega_h^e), \label{eqn_GPExt_L2_elliptic} \\[1ex]
    a_{1,h}^\text{ext}(\psih, \psih)  &\gtrsim& \| \psih \|_{1, \Omega_h^e}^2 \quad &\text{for all }& \psih \in V_h(\Omega_h^e). \label{eqn_GPExt_H1_elliptic}
\end{math}
Thus, the bilinear forms $a_h^\text{ext}(\cdot, \cdot)$ and $a_{1,h}^\text{ext}(\cdot, \cdot)$ are elliptic on $V_h(\Omega_h^e)$ and the equations \eqref{eqn_def_GPExt_problem_L2} and \eqref{eqn_def_GPExt_problem_H1} are uniquely solvable.

For the error analysis we assume a  $\phi \in H^{k+1}(\Omega_h^e)$, and $\tilde \phi $ used in \eqref{eqn_def_GPExt_problem} is meant to be an approximation of this $\phi$ on $\Omega_h^p$. Since $\phi$ is defined on $\Omega_h^e$, the error $\phi - \phih$ is well defined on $\Omega_h^e$.

\begin{theorem}\label{thm_error_GPExt_L2}
    Let $\phi \in H^{k+1}(\Omega_h^e)$ be given and $\phih \in V_h(\Omega_h^e)$ the solution of \eqref{eqn_def_GPExt_problem_L2}. The following stability and error estimates hold
    \begin{align}
      \|\phi_h\|_a & \leq \left( \|\tilde \phi\|_{\Omega_h^p}^2-s_h^{(0)}(\phi_h,\phi_h)\right)^\frac12 \leq\|\tilde \phi\|_{\Omega_h^p} \label{stabest} \\
        \| \phi - \phih \|_{\Omega_h^e} & \lesssim \|\phi - \tilde \phi\|_{\Omega_h^p} + h^{k+1} \| \phi \|_{H^{k+1}(\Omega_h^e)} \label{err1} \\
        \| \phi - \phih \|_{1,\Omega_h^e} &  \lesssim h^{-1} \|\phi - \tilde \phi\|_{\Omega_h^p} +  h^{k} \| \phi \|_{H^{k+1}(\Omega_h^e)}. \label{err2}
    \end{align}
\end{theorem}
\begin{proof}
    Let $\phi_h \in V_h(\Omega_h^e)$ be the solution of problem \eqref{eqn_def_GPExt_problem_L2}.  The stability estimate follows from
    \begin{equation*}
        \|\phi_h\|_a^2= a_h^{\rm ext}(\phi_h,\phi_h) = (\tilde \phi, \phi_h)_{\Omega_h^p} \leq \|\tilde \phi\|_{\Omega_h^p}\|\phi_h\|_{\Omega_h^p} \leq \frac12 \|\tilde \phi\|_{\Omega_h^p}^2 + \frac12 \left( \|\phi_h\|_a^2- s_h^{(0)}(\phi_h,\phi_h)\right).
    \end{equation*}
Let $\hat \phi_h \in V_h(\Omega_h^e)$ be such that
\[
  a_h^{\rm ext}(\hat \phi_h, \psi_h)= (\phi, \psi_h)_{\Omega_h^p} \quad \text{for all}~\psi_h \in V_h(\Omega_h^e).
\]
Hence, $\| \phi - \phih \|_{\Omega_h^e} \leq \| \phi_h - \hat \phi_h \|_{\Omega_h^e}+\|\phi - \hat \phi_h \|_{\Omega_h^e}$. We derive bounds for the two terms on the right-hand side. For the first one we use \eqref{eqn_GPExt_L2_elliptic}, $a_h^{\rm ext}(\phi_h-\hat \phi_h,\psi_h)=(\tilde \phi -\phi,\psi_h)_{\Omega_h^p}$ and the stability bound yielding
\begin{equation} \label{h1}
 \|\phi_h - \hat \phi_h \|_{\Omega_h^e} \lesssim \|\phi_h - \hat \phi_h \|_{a}  \leq \|\tilde \phi - \phi\|_{\Omega_h^p}.
\end{equation}
To bound the second term, we note that
 \begin{equation} \label{h2}
        \| \phi - \hat \phi_h \|_{\Omega_h^e} \leq  \|\phi - I_h \phi \|_{\Omega_h^e} + \|I_h \phi - \hat \phi_h \|_{\Omega_h^e} \lesssim h^{k+1} \|\phi \|_{H^{k+1}(\Omega_h^e)} + \|I_h \phi - \hat \phi_h \|_{\Omega_h^e}.
    \end{equation}
Using
    \begin{equation}\label{eqn_proof_GPExt_Galerkin}
        a_h^\text{ext}(\phi - \hat \phi_h, \psih) = s_h^{(0)}(\phi, \psih) \quad \text{for all } \psih \in V_h,
    \end{equation}
     we get
%\begin{equation*}
%        \|I_h \phi - \phi_h \|_{\Omega_h^e} \lesssim \|I_h \phi - \phih \|_{a}.
 %   \end{equation*}
 %   Note
    \begin{math*}[RCLCL]
        && \|I_h \phi - \hat \phi_h \|_{a}^2 =  a_h^\text{ext}(I_h \phi - \phi, I_h \phi - \hat \phi_h) + a_h^\text{ext}(\phi - \hat \phi_h, I_h\phi - \hat \phi_h) \\[1ex]
        &\overset{\eqref{eqn_proof_GPExt_Galerkin}}{=}& a_h^\text{ext}(I_h \phi - \phi, I_h \phi - \hat \phi_h) + s_h^{(0)}(\phi, I_h\phi - \hat \phi_h) \leq \| I_h \phi - \phi\|_a \| I_h \phi - \hat \phi_h \|_a + \sqrt{s_h^{(0)}(\phi, \phi)} \| I_h\phi - \hat \phi_h \|_a.
    \end{math*}
Reducing by $ \| I_h\phi - \hat \phi_h \|_a$ we get
\begin{equation}\label{aux380}
\|I_h \phi - \hat \phi_h \|_{a} \leq  \| I_h \phi - \phi\|_a + \sqrt{s_h^{(0)}(\phi, \phi)}.
\end{equation}
With the help of \eqref{eqn_from_LO19_GP_interpolation_estimate} and \eqref{eqn_from_LO19_GP_continuity_estimate} we obtain
    \begin{equation} \label{h3} %\begin{split}
        \|I_h \phi - \hat \phi_h \|_{a} \leq \| I_h \phi - \phi\|_{\Omega_h^p} + \sqrt{s_h^{(0)}(I_h \phi - \phi, I_h \phi - \phi)} + \sqrt{s_h^{(0)}(\phi, \phi)}
        \lesssim  h^{k+1} \|\phi\|_{H^{k+1}(\Omega_h^e)}. %+ \sqrt{s_h^{(0)}(I_h \phi - \phi, I_h \phi - \phi)} + \sqrt{s_h^{(0)}(\phi, \phi)}\\[1ex]
        % &\overset{\eqref{eqn_from_LO19_GP_interpolation_estimate}}{\lesssim}& h^{k+1} \|\phi\|_{H^{k+1}(\Omega_h^p)} + h^{k+1} \|\phi\|_{H^{k+1}(\Omega_h^e)}  + \sqrt{s_h^{(0)}(\phi, \phi)}\\[1ex]
        % &\overset{\eqref{eqn_from_LO19_GP_continuity_estimate}}{\lesssim}& h^{k+1} \|\phi\|_{H^{k+1}(\Omega_h^p)} + h^{k+1} \|\phi\|_{H^{k+1}(\Omega_h^e)} \lesssim h^{k+1} \|\phi\|_{H^{k+1}(\Omega_h^e)}.
    %\end{split}
    \end{equation}
    Combining this estimate with the results in \eqref{h1} and \eqref{h2} yields the bound \eqref{err1}. For the estimate in the $H^1$-norm we use the triangle inequality $\|\phi- \phi_h\|_{1,\Omega_h^e} \leq \|\phi- \hat \phi_h\|_{1,\Omega_h^e} + \|\hat \phi_h- \phi_h\|_{1,\Omega_h^e}$. For the first term we get, using an interpolation property of $I_h\phi$, a finite element inverse inequality, estimates \eqref{eqn_GPExt_L2_elliptic} and \eqref{h3},
    \begin{align*}
        \| \phi - \hat \phi_h \|_{1,\Omega_h^e}  & \leq \| \phi - I_h \phi  \|_{1,\Omega_h^e} + \| I_h \phi -  \hat \phi_h \|_{1,\Omega_h^e} \lesssim h^{k} \|\phi\|_{H^{k+1}(\Omega_h^e)} + h^{-1} \| I_h \phi - \hat \phi_h \|_{\Omega_h^e}  \\
        & \lesssim h^{k} \|\phi\|_{H^{k+1}(\Omega_h^e)} + h^{-1} \| I_h \phi - \hat \phi_h \|_a \lesssim h^{k} \|\phi\|_{H^{k+1}(\Omega_h^e)}.
    \end{align*}
For the second term we use \eqref{h1}:
\begin{equation*}
  \|\hat \phi_h- \phi_h\|_{1,\Omega_h^e} \lesssim h^{-1} \|\hat \phi_h- \phi_h\|_{\Omega_h^e} \lesssim h^{-1}\|\tilde \phi- \phi\|_{\Omega_h^p}.
\end{equation*}
Combining these results we get the bound \eqref{err2}.
\end{proof}

Now, we give an error analysis for the extension problem \eqref{eqn_def_GPExt_problem_H1}.

\begin{theorem}\label{thm_error_GPExt_H1}
    Let $\phi \in H^{k+1}(\Omega_h^e)$ be given and $\phih \in V_h(\Omega_h^e)$ the solution of \eqref{eqn_def_GPExt_problem_H1}. The following stability and error estimates hold
    \begin{align}
        \|\phih\|_{1,a} & \leq \left( \|\tilde \phi\|_{1,\Omega_h^p}^2-s_h^{(2)}(\phi_h,\phi_h)\right)^\frac12 \leq\|\tilde \phi\|_{1, \Omega_h^p} \label{stab2} \\ \label{error1}
        \| \phi - \phih \|_{1,\Omega_h^e} & \lesssim \|\phi - \tilde \phi\|_{1, \Omega_h^e} +   h^{k} \| \phi \|_{H^{k+1}(\Omega_h^e)}.
    \end{align}
\end{theorem}
\begin{proof}
The proof uses the same arguments as the proof of Theorem~\ref{thm_error_GPExt_L2}.
We define $\hat \phi_h$  as the solution of \eqref{eqn_def_GPExt_problem_H1} with $\tilde \phi$ replaced by $\phi$.
Similar to \eqref{h1} we obtain
\[
  \|\phi_h - \hat \phi_h\|_{1, \Omega_h^e} \lesssim \|\phi_h - \hat \phi_h\|_{1,a} \leq
  \|\tilde \phi - \phi\|_{1, \Omega_h^p}.
\]
Applying the triangle inequality, an interpolation property of $I_h \phi$ and \eqref{eqn_GPExt_H1_elliptic}, we obtain the estimate:
    \begin{equation*}
        \|\phi-\hat \phi_h\|_{1, \Omega_h^e} \leq \|\phi - I_h \phi\|_{1, \Omega_h^e} + \|I_h \phi-\hat  \phi_h\|_{1, \Omega_h^e} \lesssim h^{k} \|\phi\|_{H^{k+1}(\Omega_h^e)} + \|I_h \phi- \hat \phi_h\|_{1,a}.
    \end{equation*}
 Repeating the estimates \eqref{eqn_proof_GPExt_Galerkin}--\eqref{aux380} with $a_h^\text{ext}$ and $\|\cdot\|_a$ replaced by $a_{1,h}^\text{ext}$ and $\|\cdot\|_{1,a}$, respectively, we obtain
    \begin{equation*}
        \|I_h \phi-\hat \phi_h\|_{1,a} \leq \|I_h \phi-\phi\|_{1,a} + \sqrt{s_h^{(2)}(\phi, \phi)}.
    \end{equation*}
 With the help of  \eqref{eqn_from_LO19_GP_continuity_estimate} and \eqref{eqn_from_LO19_GP_interpolation_estimate} we obtain
    \begin{math*}
        \|I_h \phi-\hat \phi_h\|_{1,a} &\leq& \|I_h \phi-\phi\|_{1, \Omega_h^p} + \sqrt{s_h^{(2)}(I_h\phi - \phi, I_h\phi - \phi)} + \sqrt{s_h^{(2)}(\phi, \phi)} \\
        &\leq& h^{k} \|\phi\|_{H^{k+1}(\Omega_h^e)},
    \end{math*}
 and so the desired estimate \eqref{error1}.
\end{proof}

%\todo{AR: not able to derive an optimal $L^2$ error bound in case of $H^1$-projection}

\begin{remark} \rm \label{rem_comparison_GPext}
    We compare the results of the extension using the $L^2$- and the $H^1$-projection. Let $\tilde\phi \in V_h(\Omega_h^e)$ be a given finite element approximation of $\phi \in H^{k+1}(\Omega_h^e)$. We denote the numerical solutions of \eqref{eqn_def_GPExt_problem_L2} and \eqref{eqn_def_GPExt_problem_H1} with this $\tilde\phi$ by $\phi_h, \phi_{1,h} \in V_h(\Omega_h^e)$, respectively and the errors as $e_h := \tilde\phi - \phih$ and $e_{1,h} := \tilde\phi - \phi_{1,h}$. Note that $e_h, e_{1,h} \in V_h(\Omega_h^e)$.  It holds
    \begin{equation}\label{eqn_remark_comparison_gpext_aux}
        a_h^\text{ext}(e_h, \psi_h) = s^{(0)}(\tilde\phi, \psih) = h^2 s^{(2)}(\tilde\phi, \psih) = h^2 a^\text{ext}_{1,h}(e_{1,h}, \psih), \quad \text{for all } \psih \in V_h(\Omega_h^e).
    \end{equation}
    Using an inverse inequality we get $
            h^2 a_{1,h}^\text{ext}(w_h, w_h) \leq c\, a_h^\text{ext}(w_h, w_h),
        $
        for all $w_h \in V_h(\Omega_h^e)$.
    Using $\psih = e_h$ in \eqref{eqn_remark_comparison_gpext_aux} we get \[
        \|e_h \|_a^2 = h^2 a_{1,h}^\text{ext}(e_{1,h}, e_h) \leq c\, \|e_{1,h}\|_a \| e_h\|_a
    \]
    and thus $\|\tilde \phi - \phi_h\|_a \leq c\, \|\tilde \phi - \phi_{1,h}\|_a$ holds.
    Hence, measured in the energy norm $\|\cdot\|_a$, apart from a constant $c$, the error using the $L^2$-projection is always smaller than the error using the $H^1$-projection.   \\
%\todo{AR: discuss with Paul; MO: In real experiments, $H^1$-based extension seems to be slightly more accurate. How consistent is it with the present comment.}
\end{remark}

We present results of numerical experiments with the extension method that confirm the optimal approximation properties in section~\ref{secNumExtension}. The numerical experiments also show that there are no significant differences between the methods based on $L^2$- and $H^1$-projection.

\subsection{Long-time behavior.}

We also examine another stability aspect of the extension method. Since the extension procedure is applied in \emph{each} time step of the narrow band algorithm (cf. Section~\ref{sec_transport_overview}), it is important to assess its behavior when applied repeatedly ($n$ times). If $n$ is large, this provides insight into the long-term behavior of the method.

We restrict the discussion here to the $L^2$-projection. Similar results hold for the method based on $H^1$-projection.

For a smooth given $\phi$, consider  $\phi_h^0:=I_h \phi$ and define a sequence of level set function approximations $(\phi_h^n)_{n \in \N}$ as follows: Given $\phi_h^{n} \in V_h(\Omega_h^p)$ define $\phi_h^{n+1} \in V_h(\Omega_h^e)$ as the solution to the problem:
\begin{equation} \label{repext}
	a_h^{\rm ex}(\phi_h^{n+1}, \psi_h) =(\phi_h^{n+1}, \psi_h)_{\Omega_h^p} + s_h^{(0)}(\phi_h^{n+1},\psi_h)=(\phi_h^{n}, \psi_h)_{\Omega_h^p} \quad \text{for all}~\psi_h \in V_h(\Omega_h^e), \quad n=0,1,2, \ldots
\end{equation}
To ease the notation we drop the subscript $h$ below and write $\phi^n=\phi_h^n$,  $s(\cdot,\cdot)$ instead of $s_h^{(0)}(\cdot,\cdot)$, and  $\|\psi\|_s:=s(\psi,\psi)^\frac12$.

Note the following  identities:
\begin{align}
	(\phi^{n+1}- \phi^n,\psi)_{\Omega_h^p} +s(\phi^{n+1}, \psi) =0 \quad \text{and}~\quad
	(\phi^{n}- \phi^0,\psi)_{\Omega_h^p} &= - \sum_{k=1}^n s(\phi^k,\psi)\quad  \label{basic2}
\end{align}
which hold for all $\psi \in V_h(\Omega_h^e)$. The following lemma proves two main long-time stability results

\begin{lemma} \label{LemmaRepeated}
	%For the sequence $(\phi^n)_{n \in \Bbb{N}}$
	 The following bounds hold with a constant $c$ independent of $n$ and $h$:
	\begin{align}
		\|\phi^n\|_{\Omega_h^p} & \leq \|\phi^0\|_{\Omega_h^p}, \label{Est1} \\
%		\|\phi^n\|_{s} & \leq \|\phi^0\|_{s} \leq  c h^{k+1} \|\phi\|_{H^{k+1}(\Omega_h^e)} \label{Est2} \\
		\|\phi^n - \phi^0\|_{\Omega_h^p} &\leq c  \min\{1,\sqrt{n} h^{k+1}\} \|\phi\|_{H^{k+1}(\Omega_h^e)}.
		%\sqrt{2n} \|\phi^0\|_s %\leq c \sqrt{n} \, h^{k+1}|\Omega^e_h|^\frac12.
		\label{Est3} %\\
%		\|\phi^n - \phi^0\|_{\Omega_h^p} & \leq 2 \|\phi^0\|_{\Omega_h^p}. \label{Est4}
	\end{align}
\end{lemma}
\begin{proof}
Letting $\psi=\phi^{n+1}$ in \eqref{repext} we obtain
\begin{equation*}
	\|\phi^{n+1}\|_{\Omega_h^p}^2 + \|\phi^{n+1}\|_{s}^2 \leq \|\phi^{n}\|_{\Omega_h^p}\|\phi^{n+1}\|_{\Omega_h^p} \leq \tfrac12 \|\phi^{n}\|_{\Omega_h^p}^2 +\tfrac12 \|\phi^{n+1}\|_{\Omega_h^p}^2,
\end{equation*}
which yields \eqref{Est1}. If we let  $\psi= \phi^{n+1}-\phi^n$ in the first identity in \eqref{basic2}, we get
\begin{equation*}
	\|\phi^{n+1}- \phi^n \|_{\Omega_h^p}^2+ s(\phi^{n+1}, \phi^{n+1})= s(\phi^{n+1}, \phi^{n}) \leq \tfrac12\|\phi^{n+1}\|_s^2 + \tfrac12\|\phi^{n}\|_s^2.
\end{equation*}
This yields $\|\phi^{n+1}\|_s \leq \|\phi^{n}\|_s$  and  so $\|\phi^n\|_{s}  \leq \|\phi^0\|_{s}$.  %thus the first inequality in \eqref{Est2}.
Hence, properties \eqref{eqn_from_LO19_GP_continuity_estimate} and \eqref{eqn_from_LO19_GP_interpolation_estimate} imply
\begin{equation}\label{Est2}
	\|\phi^n\|_{s}  \leq \|\phi^0\|_s \leq \|I_h \phi - \phi\|_s + \|\phi\|_s \leq c h^{k+1} \|\phi\|_{H^{k+1}(\Omega_h^e)}.
\end{equation}
%and thus the second estimate in \eqref{Est2}.
Using $\psi= \phi^n - \phi^0$ in \eqref{basic2} and $\|\phi^{n+1}\|_s \leq \|\phi^{n}\|_s$, we get
\begin{equation*}
	\|\phi^n - \phi^0\|_{\Omega_h^p}^2 \leq n \max_{1 \leq k \leq n} \|\phi^k\|_s (\|\phi^n\|_s +\|\phi^0\|_s) \leq 2n\|\phi^0\|_s^2.
\end{equation*}
Combining this bound with \eqref{Est2} yields \eqref{Est3} with $c\sqrt{n} h^{k+1} \|\phi\|_{H^{k+1}(\Omega_h^e)}$ at the right-hand side.
This can be complemented with $	\|\phi^n - \phi^0\|_{\Omega_h^p}  \leq 2 \|\phi^0\|_{\Omega_h^p}$, which follows from a triangle inequality and \eqref{Est1}. Finally, note that $\|\phi^0\|_{\Omega_h^p}=\|I_h \phi\|_{\Omega_h^p} \leq c \|\phi\|_{H^{k+1}(\Omega_h^e)}$ holds.
\end{proof}

%It follows from  the result \eqref{Est4} that  the estimate with $\sqrt{n}$ in \eqref{Est3} is too pessimistic for $n \to \infty$. The reason for this is that $\lim_{n \to \infty} \phi^n = \phi^\ast$ for a $\phi^\ast \in \ker(s)$ holds, i.e. $\lim_{k \to \infty} s(\phi^k, \psi)=0$ in \eqref{basic2}, cf. Appendix.

The result in \eqref{Est3} is likely sharp in terms of dependence on $n$, since the  initial growth of $\|\phi^n - \phi^0\|_{\Omega_h^p}$ with $\sqrt{n}$ is observed in numerical experiments, cf. Section~\ref{secNumExtension}.

We note that the choice $\phi_h^0:=I_h \phi$ slightly simplified  the analysis, but  causes no loss of generality. Indeed, if $\tilde \phi^{n}$ is the sequence defined as in \eqref{repext}, with some general $\phi_h^0= \tilde \phi$, then using \eqref{stabest} we obtain that $\tilde \phi^n$ and $\phi^n$ stay uniformly in $n$ close in the following sense:
\begin{equation*}
    \|\tilde \phi^n -  \phi^n\|_a \leq \|\tilde \phi - I_h \phi\|_{\Omega_h^p} \leq \|\tilde \phi - \phi\|_{\Omega_h^p} + ch^{k+1} \|\phi\|_{H^{k+1}(\Omega_h^p)}.
\end{equation*}

\subsection{Numerical experiments with the extension method.} \label{secNumExtension}

We present results of numerical experiments with the ghost penalty extension method defined in  \eqref{eqn_def_GPExt_problem}. We choose the exact level set function as $\phi(\bx) := (x_1 - x_3^2)^2 + x_2^2 + x_3^2 - 1$. The corresponding zero level defines a so-called ``Kite''-geometry, which was also used in  \cite{BrandnerReusken2020, Dziuk1988}. See Figure \ref{fig_KiteToSphere_start} for a visualization of the surface. The computational domain is chosen as $\Omega = [-\frac53,\frac53]^3$ and the initial mesh is a uniform tetrahedral mesh with  mesh size $h_0 = 0.5$. In each refinement step we halve the mesh size. As in  the error analysis, we use the parameter values $\alpha = 0$ for the $L^2$-projection and $\alpha = 2$ for the $H^1$-projection. The coefficient $\gamma^\text{ext}$ is chosen as $\gamma^\text{ext} = 1$. We define the domain $\Omega_h^p$ as the elements cut by the surface, denoted by  $\cT_\G$, with two layers of neighbors added, i.e., $\Omega_h^p = \cN^2(\cT_\G)$. As the input function $\tilde \phi$, we take an Oswald-type interpolation of $\phi$ in $V_h(\Omega_h^p)$. For the polynomial degree in the finite element space $V_h$ we consider $k = 1$ or $k = 2$. The extension elements are chosen as one and two layers of neighbors of $\Omega_h^p$. Thus, it holds $\Omega_h^e = \cN^1(\Omega_h^p)$ or $\Omega_h^e = \cN^2(\Omega_h^p)$.

We investigate the $L^2$- and the $H^1$-errors for the $L^2$- and the $H^1$-projection. We present results for the errors
\begin{math*}[RCLRCL]
    \left(e_\text{ext}\right)^2 &:=& \oint_{\Omega_h^e} & \left(\phi - \phi_h^{L^2}\right)^2 \dx, \qquad \left(e^\nabla_\text{ext}\right)^2 &:=& \oint_{\Omega_h^e} \nabla\left(\phi - \phi_h^{L^2}\right)^2 \dx, \\
    \left(e_\text{ext}^1\right)^2 &:=& \oint_{\Omega_h^e} & \left(\phi - \phi_h^{H^1}\right)^2 \dx, \qquad \left(e^{1,\nabla}_\text{ext}\right)^2 &:=& \oint_{\Omega_h^e} \nabla\left(\phi - \phi_h^{H^1}\right)^2 \dx,
\end{math*}
where $\phih^{L^2}$ and $\phih^{H^1}$ denote the solutions of the ghost penalty extension using the $L^2$- or the $H^1$-projection, respectively. Note that the areas of the domains $\Omega_h^p$ and $\Omega_h^e$ depend on $h$. Therefore, for the errors we  use the scaled integrals given by $\oint_{\Omega_h^e} f \dx = \int_{\Omega_h^e} f \dx / \|1\|_{\Omega_h^e}^2$.

The results of the extension to one layer ($\Omega_h^e= \cN^1(\Omega_h^p)$) are shown in Figure \ref{fig_GPExtOneLayer} and for two layers ($\Omega_h^e= \cN^2(\Omega_h^p)$) in Figure \ref{fig_GPExtTwoLayers}. We observe convergence rates $\cO(h^2)$ for the $L^2$-error and $\cO(h)$ for the $H^1$-error in the case  $k = 1$. For $k = 2$ we observe convergence rates  $\cO (h^3)$ for the $L^2$-error and $\cO( h^2)$ for the $H^1$-error. These (optimal) orders of convergence are in agreement with the error analysis in Theorems \ref{thm_error_GPExt_L2} and \ref{thm_error_GPExt_H1}. Note, that the errors for the $L^2$-projection behave very similar to those of the  $H^1$-projection.
%This may be related to the property discussed in Remark \ref{rem_comparison_GPext}.
In this Kite example, the exact level set function is a polynomial of degree four. Choosing $k = 4$ in the extension method leads to errors close to the machine accuracy $(\sim 10^{-11})$ for the $L^2-$ and the $H^1-$projection.

\newcommand{\errorfilelayeroneorderone}{./errors/GPExtension_Kite/order1_LayersExt1_errors.dat}
\newcommand{\errorfilelayeroneordertwo}{./errors/GPExtension_Kite/order2_LayersExt1_errors.dat}

\begin{figure}[!ht]
    \centering
    \begin{subfigure}[b]{0.4\textwidth}
        \begin{tikzpicture}
            \pgfplotsset{legend style={at={(0.99,0.01)}, anchor=south east, legend columns=1, draw=none, fill=none},}
            \begin{axis}[domain={0.0305:0.55}, ymode=log, xmode=log]
                \addplot[red, mark=*] table[x=MeshSize, y=L2ProjL2OnExtElem]{\errorfilelayeroneorderone};
                \addlegendentry{\small $e_\text{ext}$}
                \addplot[blue, mark=*] table[x=MeshSize, y=H1ProjL2OnExtElem]{\errorfilelayeroneorderone};
                \addlegendentry{\small $e^1_\text{ext}$}
                \addplot[dotted, red, mark=*] table[x=MeshSize, y=L2Projh1OnExtElem]{\errorfilelayeroneorderone};
                \addlegendentry{\small $e^\nabla_\text{ext}$}
                \addplot[dotted, blue, mark=*] table[x=MeshSize, y=H1Projh1OnExtElem]{\errorfilelayeroneorderone};
                \addlegendentry{\small $e^{1, \nabla}_\text{ext}$}
                \addplot[dashed,line width=0.75pt]{3.2*x};
                \addlegendentry{\small $\cO(h)$}
                \addplot[dotted,line width=0.75pt]{1.6*x^2};
                \addlegendentry{\small $\cO(h^2)$}
            \end{axis}
        \end{tikzpicture}
        \caption{$k = 1$}
    \end{subfigure}
    \begin{subfigure}[b]{0.4\textwidth}
        \begin{tikzpicture}
            \pgfplotsset{legend style={at={(0.99,0.01)}, anchor=south east, legend columns=1, draw=none, fill=none},}
            \begin{axis}[domain={0.0305:0.55}, ymode=log, xmode=log]
                \addplot[red, mark=*] table[x=MeshSize, y=L2ProjL2OnExtElem]{\errorfilelayeroneordertwo};
                \addlegendentry{\small $e_\text{ext}$}
                \addplot[blue, mark=*] table[x=MeshSize, y=H1ProjL2OnExtElem]{\errorfilelayeroneordertwo};
                \addlegendentry{\small $e^1_\text{ext}$}
                \addplot[dotted, red, mark=*] table[x=MeshSize, y=L2Projh1OnExtElem]{\errorfilelayeroneordertwo};
                \addlegendentry{\small $e^\nabla_\text{ext}$}
                \addplot[dotted, blue, mark=*] table[x=MeshSize, y=H1Projh1OnExtElem]{\errorfilelayeroneordertwo};
                \addlegendentry{\small $e^{1, \nabla}_\text{ext}$}
                \addplot[dashed,line width=0.75pt]{1.8*x^2};
                \addlegendentry{\small $\cO(h^2)$}
                \addplot[dotted,line width=0.75pt]{1.0*x^3};
                \addlegendentry{\small $\cO(h^3)$}
            \end{axis}
        \end{tikzpicture}
        \caption{$k = 2$}
    \end{subfigure}
    \caption{Ghost penalty extension of the level set function on one extension layer, $\Omega_h^e = \cN^1(\Omega_h^p)$.} \label{fig_GPExtOneLayer}
\end{figure}

\newcommand{\errorfilelayertwoorderone}{./errors/GPExtension_Kite/order1_LayersExt2_errors.dat}
\newcommand{\errorfilelayertwoordertwo}{./errors/GPExtension_Kite/order2_LayersExt2_errors.dat}

\begin{figure}[!ht]
    \centering
    \begin{subfigure}[b]{0.4\textwidth}
        \begin{tikzpicture}
            \pgfplotsset{legend style={at={(0.99,0.01)}, anchor=south east, legend columns=1, draw=none, fill=none},}
            \begin{axis}[domain={0.0305:0.55}, ymode=log, xmode=log]
                \addplot[red, mark=*] table[x=MeshSize, y=L2ProjL2OnExtElem]{\errorfilelayertwoorderone};
                \addlegendentry{\small $e_\text{ext}$}
                \addplot[blue, mark=*] table[x=MeshSize, y=H1ProjL2OnExtElem]{\errorfilelayertwoorderone};
                \addlegendentry{\small $e^{1}_\text{ext}$}
                \addplot[dotted, red, mark=*] table[x=MeshSize, y=L2Projh1OnExtElem]{\errorfilelayertwoorderone};
                \addlegendentry{\small $e^\nabla_\text{ext}$}
                \addplot[dotted, blue, mark=*] table[x=MeshSize, y=H1Projh1OnExtElem]{\errorfilelayertwoorderone};
                \addlegendentry{\small $e^{1, \nabla}_\text{ext}$}
                \addplot[dashed,line width=0.75pt]{4.2*x};
                \addlegendentry{\small $\cO(h)$}
                \addplot[dotted,line width=0.75pt]{4.5*x^2};
                \addlegendentry{\small $\cO(h^2)$}
            \end{axis}
        \end{tikzpicture}
        \caption{$k = 1$}
    \end{subfigure}
        \begin{subfigure}[b]{0.4\textwidth}
        \begin{tikzpicture}
            \pgfplotsset{legend style={at={(0.99,0.01)}, anchor=south east, legend columns=1, draw=none, fill=none},}
            \begin{axis}[domain={0.0305:0.55}, ymode=log, xmode=log]
                \addplot[red, mark=*] table[x=MeshSize, y=L2ProjL2OnExtElem]{\errorfilelayertwoordertwo};
                \addlegendentry{\small $e_\text{ext}$}
                \addplot[blue, mark=*] table[x=MeshSize, y=H1ProjL2OnExtElem]{\errorfilelayertwoordertwo};
                \addlegendentry{\small $e^1_\text{ext}$}
               \addplot[dotted, red, mark=*] table[x=MeshSize, y=L2Projh1OnExtElem]{\errorfilelayertwoordertwo};
                \addlegendentry{\small $e^\nabla_\text{ext}$}
                \addplot[dotted, blue, mark=*] table[x=MeshSize, y=H1Projh1OnExtElem]{\errorfilelayertwoordertwo};
                \addlegendentry{\small $e^{1, \nabla}_\text{ext}$}
                \addplot[dashed,line width=0.75pt]{5.9*x^2};
                \addlegendentry{\small $\cO(h^2)$}
                \addplot[dotted,line width=0.75pt]{4.7*x^3};
                \addlegendentry{\small $\cO(h^3)$}
            \end{axis}
        \end{tikzpicture}
        \caption{$k = 2$}
    \end{subfigure}
    \caption{Ghost penalty extension of the level set function on two extension layers, $\Omega_h^e = \cN^2(\Omega_h^p)$.} \label{fig_GPExtTwoLayers}
\end{figure}

We now consider the repeated application of the extension method, cf. Lemma~\ref{LemmaRepeated} and the estimate \eqref{Est3}.   We start with the initial function $\phi_h^0=I_h \phi$ and determine $\phi_h^{n+1} \in V_h(\Omega_h^e)$ such that
\begin{equation*}
    (\phi_h^{n+1}, \psi_h)_{\Omega_h^p} + s_h^{(0)}(\phi_h^{n+1},\psi_h)=(\phi_h^{n}, \psi_h)_{\Omega_h^p} \quad \text{for all}~\psi_h \in V_h(\Omega_h^e), \quad n=0,1,2, \ldots
\end{equation*}
holds. As in the experiments above we use the Kite-geometry and choose $\Omega_h^p=\cN^2(\cT_\G)$ as the cut elements with two layers of neighbors added and  $\Omega_h^e = \cN^2(\Omega_h^p)$. We apply the extension method using the $L^2$-projection with $\alpha = 0$ and $\gamma^\text{ext} = 1$. The polynomial degrees in the finite element space are chosen as $k = 1$ and $k = 2$, and we present results for the  error quantities
\begin{equation*}
    e_{\Omega_h^p} = \frac{\|\phi_h^n - \phi_h^0 \|_{\Omega_h^p}}{\|1\|_{\Omega_h^p}}, \quad e_{\Omega_h^e} = \frac{\|\phi_h^n - \phi_h^0 \|_{\Omega_h^e}}{\|1\|_{\Omega_h^e}},
%\quad e_s := \frac{s(\phi_h^n,\phi_h^n)^\frac12}{\|1\|_{\Omega_h^e}}.
\end{equation*}

The results for fixed mesh size $h=2^{-4}$ and $0 \leq n \leq 1000$  are shown in Figure \ref{fig_GPExtIterativ}. For $k=1$ we observe for $e_{\Omega_h^p}$ a $\sqrt{n}$ dependence on $n$, consistent with \eqref{Est3}. For $k=2$ the dependence is slightly milder.

\newcommand{\errfileFixHOne}{./errors/GPExtension_Kite/iterativ_order1_errors.dat}
\newcommand{\errfileFixHTwo}{./errors/GPExtension_Kite/iterativ_order2_errors.dat}

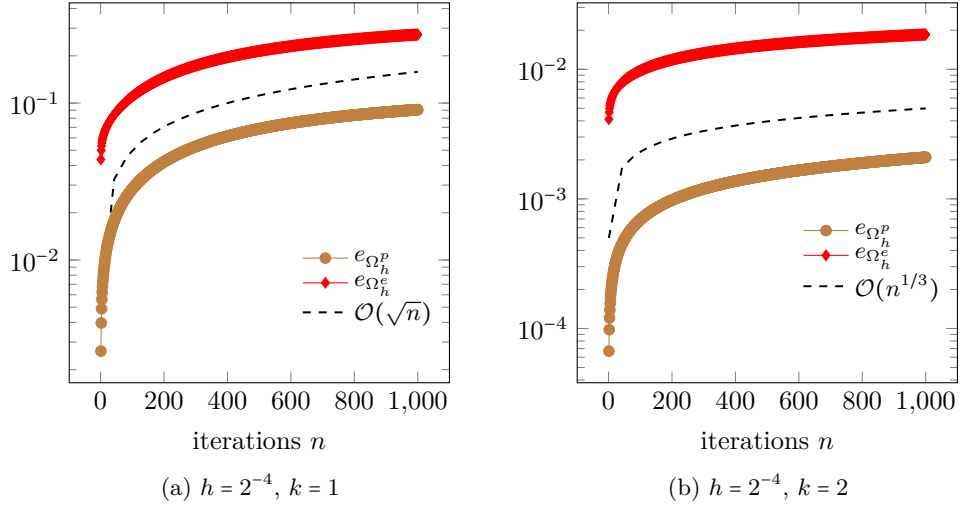
\begin{figure}[!ht]
    \centering
    \begin{subfigure}[b]{0.4\textwidth}
        \begin{tikzpicture}
            \pgfplotsset{legend style={at={(0.99,0.05)}, anchor=south east, legend columns=1, draw=none, fill=none, xlabel=iterations $n$},}
            \begin{axis}[domain={1:1000}, xmode=log, ymode=log]
                \addplot[brown, mark=*] table[x=Iterations, y=L2OnProjElem_Lx3]{\errfileFixHOne}; \addlegendentry{\small $e_{\Omega_h^p}$}
                \addplot[red, mark=diamond*] table[x=Iterations, y=L2OnExtElem_Lx3]{\errfileFixHOne}; \addlegendentry{\small $e_{\Omega_h^e}$}
                \addplot[black, dashed,line width=0.75pt]{0.005*x^0.5}; \addlegendentry{\small $\cO(\sqrt{n})$}
            \end{axis}
        \end{tikzpicture}
        \caption{$h = 2^{-4}$, $k = 1$}
    \end{subfigure}
    \begin{subfigure}[b]{0.4\textwidth}
        \begin{tikzpicture}
            \pgfplotsset{legend style={at={(0.99,0.05)}, anchor=south east, legend columns=1, draw=none, fill=none, xlabel=iterations $n$},}
            \begin{axis}[domain={1:1000},  xmode=log, ymode=log]
                \addplot[brown, mark=*] table[x=Iterations, y=L2OnProjElem_Lx3]{\errfileFixHTwo}; \addlegendentry{\small $e_{\Omega_h^p}$}
                \addplot[red, mark=diamond*] table[x=Iterations, y=L2OnExtElem_Lx3]{\errfileFixHTwo}; \addlegendentry{\small $e_{\Omega_h^e}$}
                \addplot[blue,  dashed, line width=0.75pt]{0.001*x^0.333}; \addlegendentry{\small $\cO(n^{1/3})$}
                \addplot[black, dashed, line width=0.75pt]{0.00015*x^0.5}; \addlegendentry{\small $\cO(\sqrt{n})$}
            \end{axis}
        \end{tikzpicture}
        \caption{$h = 2^{-4}$, $k = 2$}
    \end{subfigure}
    \caption{Fixed mesh size $h$ and varying iteration number $n$} \label{fig_GPExtIterativ}
\end{figure}

\section{A level set narrow band algorithm.}\label{sec_transport_whole_algo}

We have all ingredients ready for the narrow band level set transport as outlined in Algorithm~\ref{alg}. This section discusses some implementation aspects and presents more explanation on the method.

An important aspect of the method is the selection of the domains $\Omega_\G^n$ for solving the level set equation in the time interval $[t_n,t_{n+1}]$ and the subdomain $\Omega_h^{p,n}$ where the projection in the extension method \eqref{eqn_def_GPExt_problem} is performed. For simplicity, we assume the velocity field $\bu$ is given in a sufficiently large narrow band (cf. Remark~\ref{rem_extrapolation_vel_transport}).

Recall that the domain formed by the simplices cut by $\Gamma^n$ is denoted by $\cT_\G^n$. Since the exact zero level of the level set function is not available, $\cT_\G^n$ is determined using the finite element approximation $\phi_h(\cdot,t_n)$ of $\phi(\cdot,t_n)$. The narrow band domain $\Omega_{\Gamma}^n$ includes $\cT_\G^n$ plus a few layers of neighboring elements. In the experiments, three layers are added, so $\Omega_{\Gamma}^n = \cN^3(\cT_\G^n)$. The time step is chosen so that $\Gamma_h^{n+1}$ (the zero level of $\phi_h(\cdot,t_{n+1})$) is within $\Omega_{\Gamma}^n$ (see Fig.~\ref{fig_NarrowBand}). Thus, at each time step, the level set equation is solved on a narrow band domain of width $\sim ch$, with a constant $c$. The level set approximation at $t=t_{n+1}$, denoted by $\tphi_h^{n+1}$ and defined on $\Omega_\Gamma^n$, is then extended to $\phi_h^{n+1}$ defined on $\Omega_\Gamma^{n+1}$ using the extension method \eqref{eqn_def_GPExt_problem}.

Clearly the projection domain $\Omega_h^p = \Omega_h^{p, n+1}$ has to be a subdomain of both the current transport domain $\Omega_\G^n$ where $\tphi_h^{n+1}$ is given and the next transport domain $\Omega_\G^{n+1}$ where $\phi_h^{n+1}$ has to be defined. Thus, the maximal possible projection domain is $\Omega_h^{p,n+1} = \Omega_\G^n \cap \Omega_\G^{n+1}$. However, we use a  smaller projection domain. This choice of $\Omega_h^{p,n+1}$ addresses a key challenge in narrow band level set methods: obtaining accurate numerical boundary data on the inflow boundary. An inaccuracy of boundary data at $t=t_n$ on the inflow boundary $\partial \Omega_{\Gamma,in}^n$ can significantly affect the solution of the level set equation in $\Omega_\Gamma^n \times [t_n,t_{n+1}]$. However, due to the transport nature of the flow, such errors remain (essentially) on  the same level line. For instance, an error at a point $z$ on $\partial \Omega_{\Gamma,in}^n$ with $\phi(z,t_n)=ch$ is transported along a trajectory $z(t)$ with $\phi(z(t),t_n)=ch$ for $t \in [t_n,t_{n+1}]$. Therefore, these boundary errors do not enter the domain $\{z \in \Omega_\Gamma^n \mid |\phi(z,t_{n+1})| < ch\}$. Although this property does not hold exactly for the discrete approximation $\phi_h$, it is expected to hold approximately for an accurate discretization of the level set equation.

Motivated by this observation, we use a \emph{smaller} projection domain $\Omega_h^{p, n+1}$ than the maximal one $\Omega_\Gamma^n \cap \Omega_\Gamma^{n+1}$. Obvious candidates are $\cT_\G^{n+1}$ (the domain formed by elements cut by $\Gamma_h^{n+1}$), $\cN^1(\cT_{\Gamma}^{n+1})$ (the ``cut'' elements and their direct neighbors), or the smallest set of simplices containing the level set $h$-tube $\{x \in \R^d \mid |\tilde \phi_h^{n+1}(x)| \leq h\}$. In the numerical experiments, we use $\Omega_h^p = \cN^1(\cT_\G^{n+1})$, since this domain is independent of the scaling of $\phi_h$ and is easy to compute. Important effects of this choice of the projection domain are illustrated in numerical experiments in Section~\ref{sec_transport_experiments}.

Since we  use a narrow band with a width $\sim ch $, we have to adjust the time step size to ensure that the zero level of $\phi_h(\cdot,t_{n+1})$ is within the narrow band $\Omega_\G^n$. The normal velocity of the surface is given by $V_\G(x)$, $x \in \Gamma(t)$. Motivated by the maximal movement of the zero level set in one time step $(t_{n+1} - t_n) \max_{t \in [t_n, t_{n+1}]}{\|V_\G\|_{L^\infty(\Gamma(t))}}$, we use the time step size $\Delta t_n = \frac{1}{2^k} \frac{(j-1) h}{2 \|V_\G^n \|_{L^\infty}}$. Here, $k$ is the order of the used polynomials and $j$ denotes the number of layers added to $\cT_\G^n$ to define the narrow band.
%\todo{AR to PS: why this scaling with $k$?}
%\todo[color=green]{PS to AR: because otherwise we loose the higher order convergence}
%\todo{AR to PS: verstehe ich nicht; weshalb? Morgen besprechen}
In each time step we check that the projection domain $\Omega_h^{p,n+1} = \cN^1(\cT_\G^{n+1})$ is contained in the narrow band $\Omega_\G^n$. If this condition is violated, which happens only in very rare cases, we halve the time step size and redo the transport step. This procedure is repeated until the condition is satisfied.

The BDF schemes with variable time step size can be derived along the same lines as the formulas with a fix time step size. It is used and analyzed for example in \cite{AkrivisChenHanYuZhang2024, ByrneHindmarsh1987, HairerWanner1996, Rockswold1983}.

Summarizing we obtain Algorithm \ref{alg_levelset_transport}:

\begin{algorithm}[h!]
    \caption{Narrow band level set transport}{} \label{alg_levelset_transport}
	$\phih^0, \Omega_\Gamma^0 \gets $ Initialization\\
    \emph{while t < T do}:
    \begin{enumerate}
    	\item $\phi_D^n \gets$ determine boundary data  on $\partial \Omega_{\Gamma,in}^n$, e.g. \eqref{bnd_data_BDF2_lo} for BDF2
    	\item $\Delta t_n \gets$ calculate time step size (and $t_{n+1}$)
	    \item $\tphih \gets$ perform transport step on $\Omega_\Gamma^n \times [t_n,t_{n+1}]$ using \eqref{eqn_discrete_levelset_transport} and project into $V_h(\Omega_\Gamma^n)$
    	\item $\Omega_\Gamma^{n+1} \gets $ new narrow band with few layers around $\tilde \Gamma^{n+1} = \{\tphih = 0\}$
	    \item $\Omega_h^{p,n+1} \gets \cN^1(\cT_{\tilde\G}^{n+1})$
	    \item if $\Omega_h^{p,n+1} \not\subset \Omega_\Gamma^n$ then $\Delta t_n \gets \Delta t_n/2$ and go to step 3
	    \item $\phih^{n+1} \gets$  extension step \eqref{eqn_def_GPExt_problem_L2} of $\tphih$ to $\Omega_\G^{n+1}$
	\end{enumerate}
\end{algorithm}

Note that in this approach, we do \emph{not} use any re-initialization of the level set function. In numerical experiments (see Section~\ref{sec_transport_experiments}), we observed that for cases with smoothly evolving level sets, our approach appears to work satisfactorily even without re-initialization. An important feature of this approach is that higher-order approximations can be easily achieved by using a higher-order time and space discretization method for the level set equation and a higher-order finite element space in the extension method.
%For higher-order BDF schemes one simplify has to store and extend more than the last level set function.
Although no rigorous error analysis of the entire algorithm is available yet, rigorous error bounds for the three components (numerical boundary data, space and time discretization of the level set equation, and the extension method) that form the algorithm have been derived.

\begin{remark}\label{rem_extrapolation_vel_transport} \rm
In this paper, we assume the flow field $\bu$ is known in the narrow band, but this is usually not the case in practice. If the velocity field is determined by a partial differential equation on
$\Gamma(t)$ %(as in a surface Navier-Stokes equation)
and TraceFEM is used for discretization, then an approximation of the velocity field is known in a narrow band, with the width chosen in TraceFEM. In such cases, there is a strong coupling between the PDE on the surface (which determines $\bu$) and the level set equation (which depends on $\bu$ and determines the surface location). To obtain a reasonable approximation of $\bu^{n+1}$ for use in \eqref{eqn_discrete_levelset_transport}, one can extrapolate the velocity in time based on previous time steps, as in \cite{elliott2022numerical}. If the velocity is defined in a narrow band around the surface (e.g., from TraceFEM), the extension method \eqref{eqn_def_GPExt_problem} can extend it to a wider band, or an implicit extension can be used as suggested in \cite{lehrenfeld2018stabilized}.
\end{remark}

\section{Numerical experiments with narrow band algorithm} \label{sec_transport_experiments}

We give some results of numerical experiments to illustrate the performance of the narrow band level set algorithm explained in Section~\ref{sec_transport_whole_algo}.  All examples are implemented in NGSolve/netgen, cf. \cite{ngsolve} with the add-on ngsxfem, cf. \cite{ngsxfem}.

We give numerical results for the level set equation \eqref{LS3} on narrow bands in  2D and 3D. The time interval is taken as $t \in [0, 1]$ if not stated otherwise. We use the BDF$2$ scheme in time and first order polynomials ($k=1$) in space or a BDF$3$ method combined with second order polynomials ($k=2$). To get sufficiently good initial values for the BDF$2$ or BDF$3$ scheme, we initialize $\phi_h^0$ as an interpolation of the exact level set function and start with some BDF$1$ time steps with a sufficiently small time step size.

In case of the BDF2 method, for the inflow boundary condition we use \eqref{bnd_data_BDF2_lo}. If BDF3 is used, we apply an analogous time extrapolation using the values at  $t\in\{t_n,t_{n-1},t_{n-2}\}$.

If not stated otherwise, the narrow band scheme is applied on the cut elements with three layers of neighboring elements added, $\Omega_\Gamma^n := \mathcal{N}^3(\mathcal{T}_\Gamma^n)$ and for the projection domain in the extension method, we use $\Omega_h^{p, n+1} = \mathcal{N}^1(\mathcal{T}_\Gamma^{n+1})$. At each time step, we check that the projection elements for the next time step are contained in $\Omega_\Gamma^n$. For the ghost penalty extension, we use the $L^2$-projection with $\alpha = 0$ and $\gamma^\text{ext} = 1$. In each time step, the exact velocity $\bu$ is interpolated into a continuous finite element space of degree $k$ on the narrow band $\Omega_\Gamma^n$.

To evaluate the accuracy of the method, we track the following error metrics
\begin{equation*}
    |e_\G|^2 := \sum_{n = 1}^N \Delta t_n \oint_{\Gh^n} (\phi^n)^2 \dx, \qquad e_\G^\infty = \max_{1\leq n \leq N} \|\phi^n\|_{L^\infty(\G_h^n)}, \qquad |e_{L^2}|^2 := \sum_{n = 1}^N \Delta t_n \oint_{\Omega_\G^n} \left(\phi^n - \phi_h^n\right)^2 \dx,
\end{equation*}
where $\|\phi^n\|_{L^\infty(\G_h^n)}$ is an approximation of the maximal (absolute) value of the exact level set function $\phi^n$ on $\G_h^n$. Recall that $\oint_{\Omega_\G^n}$ stands for $|\Omega_\G^n|^{-1}\int_{\Omega_\G^n}$. Note, that the integrals over $\G_h^n$ cannot be computed exactly when using polynomials of degree $k > 1$, and standard quadrature rules cannot be applied. For this reason, we use the so-called parametric approach described for example in \cite{Lehrenfeld2017, GrandeLehrenfeldReusken2018} to approximate the error quantities with sufficiently high accuracy. The quantities $e_\G$ and $e_\G^\infty$ are a measure for the approximation error in $\Gamma_h^n \approx \Gamma(t_n)$.

We  present results for the following test cases:
\begin{itemize}
    \item \emph{Deforming kite}. We consider the deformation of a kite geometry to a sphere in 3D, kite($3D$), and the deformation of a kite geometry to a circle in 2D, kite($2D$), cf. Figure~\ref{fig_KiteToSphere}. We present results for the error quantities for BDF2 with $k=1$ and  BDF3 with $k=2$. Furthermore, for kite($2D$) with BDF2 and $k=1$ we illustrate the effect of different inflow boundary data.
    \item \emph{Ball in a rotating flow field}. We consider a sphere (3D) or a circle (2D) that makes a full rotation, cf. Figure~\ref{fig_TravelingSphere}. We present results for the error quantities for BDF2 with $k=1$ and  BDF3 with $k=2$.  For the 2D case we show the effect of a ``too large'' projection domain. For the 3D case we study volume conservation properties.
    \item \emph{Strongly deforming sphere}. As a more challenging and less smooth example, we consider a sphere that is stretched to a thin tube by a vortex flow field. We study different configurations and present accuracy results for BDF2 with $k=1$ and  BDF3 with $k=2$.
\end{itemize}

\subsection{Deforming kite (2D, 3D)}

The exact level set function of the kite is given by $\phi(x, t=0) = (x_1 + x_2^2)^2 + x_2^2 - 1$ in 2D and $\phi(x, t=0) = (x_1 - x_3^2)^2 + x_2^2 + x_3^2 - 1$ in 3D. We use a convex combination to deform the kite into a circle (or a sphere) with level set function $\phi(x, t=1) = x_1^2 + x_2^2 - 1$, or $\phi(x, t) = x_1^2 + x_2^2 + x_3^2 - 1$ in 2D and 3D, respectively. The geometry evolution for the 3D case is illustrated in Figure \ref{fig_KiteToSphere}.

\begin{figure}[ht!]
    \centering
    \begin{subfigure}[b]{0.25\textwidth}
        \centering
        \includegraphics[width=\textwidth]{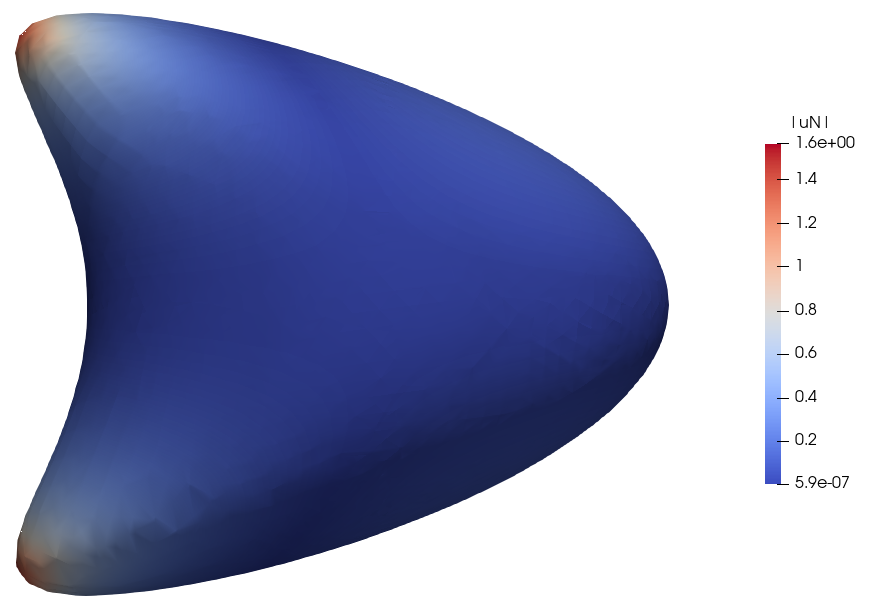}
        \caption{$t = 0$} \label{fig_KiteToSphere_start}
    \end{subfigure}
    \hspace*{1cm}
    \begin{subfigure}[b]{0.25\textwidth}
        \centering
        \includegraphics[width=\textwidth]{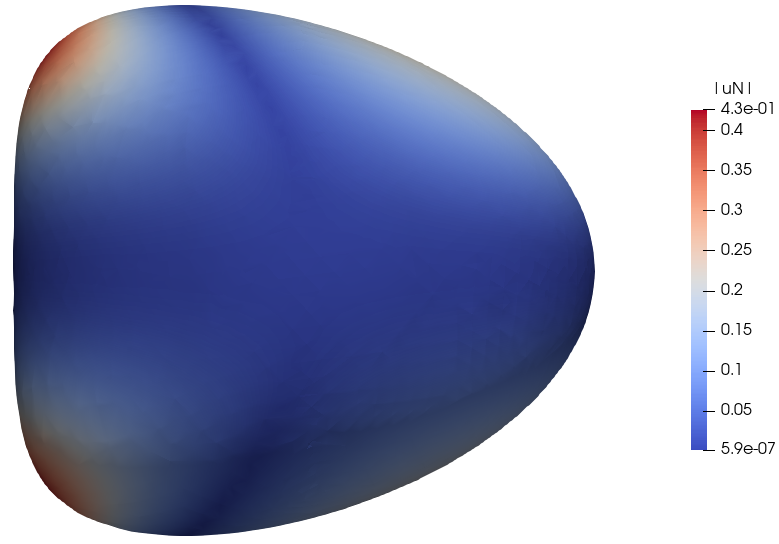}
        \caption{$t = 0.5$}
    \end{subfigure}
    \hspace*{1cm}
    \begin{subfigure}[b]{0.25\textwidth}
        \centering
        \includegraphics[width=0.85\textwidth]{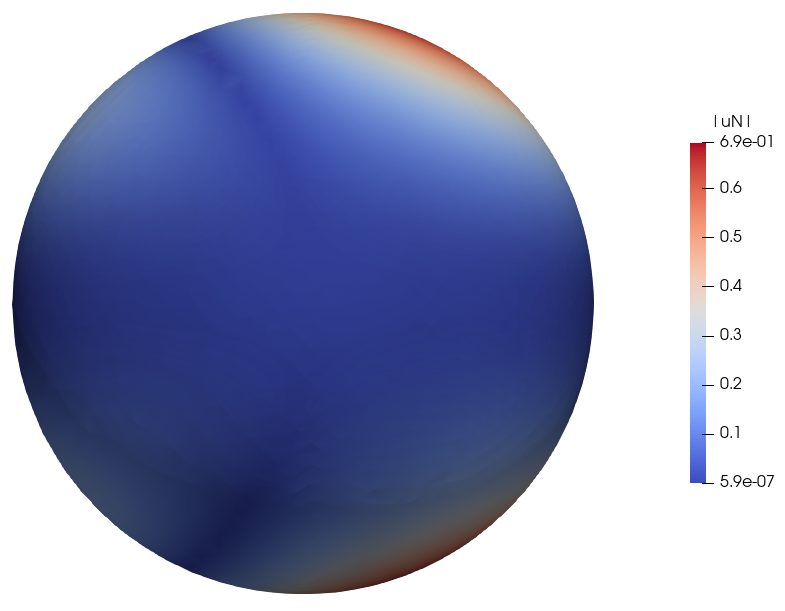}
        \caption{$t = 1$}
    \end{subfigure}
    \caption{Kite deforming into a sphere.}\label{fig_KiteToSphere}
\end{figure}

The velocity field $\bu$ used in the level set equation has a normal component that is determined by the level set evolution:  $\bu \cdot \bn = V_\G = - \frac{\partial}{\partial t} \phi / \bn \cdot \nabla \phi$.
%In applications (e.g. when the velocity is defined by a surface Navier-Stokes system), the velocity has a tangential component that should not change the geometrical evolution of the surface.
To illustrate that our approach is not restricted to cases with normal velocity fields, we add a tangential component $\bu_T$  to $\bu$.  We use $\bu_T = \mathbf{curl}_\G (x_1 x_2 x_3)$ as a tangential component
%with the curl operator defined as in \cite{BrandnerReusken2020,Reusken2018}
in 3D and the tangential component of $(y, -x)^T$ in the 2D case.

The mesh size start as $h_0 = 0.5$, and we halve it in each refinement step. The results are shown in Fig.~\ref{fig_Kite2DConv} for the 2D case and in Fig.~\ref{fig_Kite3DConv} for the 3D case. We observe optimal convergence rates of $\cO(h^2)$ for $k=1$ and $\cO(h^3)$ for $k=2$ for all error measures in two dimensions. In three dimensions, we observe convergence rates that are close to the optimal ones.

\newcommand{\ErrKiteTwoDOrderOne}{./errors/Transport2D/Kite_Order_1_SmallProj_errors.dat}
\newcommand{\ErrKiteTwoDOrderTwo}{./errors/Transport2D/Kite_Order_2_SmallProj_errors.dat}

\begin{figure}[ht!]
    \centering
    \begin{subfigure}[b]{0.40\textwidth}
        \begin{tikzpicture}
            \pgfplotsset{legend style={at={(0.99,0.01)}, anchor=south east, legend columns=1, draw=none, fill=none},}
            \begin{axis}[domain={0.0076:0.55}, ymode=log, xmode=log] % ymin=0.00002, ymax=1.25,
                \addplot[brown, mark=square*] table[x=MeshSize, y=L2L2lsetTrans]{\ErrKiteTwoDOrderOne};
                \addlegendentry{\small $e_{L^2}$}
                \addplot[blue, mark=diamond*] table[x=MeshSize, y=LinfLinfExLsetOnGh]{\ErrKiteTwoDOrderOne};
                \addlegendentry{\small $e_\G^\infty$}
                \addplot[red, mark=*] table[x=MeshSize, y=L2ExLsetOnGh]{\ErrKiteTwoDOrderOne};
                \addlegendentry{\small $e_\G$}
                \addplot[black, dashed,line width=0.75pt]{1.2*x^2};
                \addlegendentry{\small $\cO(h^2)$}
                % \addplot[red, mark=diamond*] table[x=MeshSize, y=L2L2lsetCut]{\ErrKiteTwoDOrderOne};
                % \addlegendentry{\small $e_{L^2}^{\text{Cut}}$}
            \end{axis}
        \end{tikzpicture}
        \caption{BDF$2$, $k=1$}
    \end{subfigure}
    \begin{subfigure}[b]{0.40\textwidth}
        \begin{tikzpicture}
            \pgfplotsset{legend style={at={(0.99,0.01)}, anchor=south east, legend columns=1, draw=none, fill=none},}
            \begin{axis}[domain={0.0076:0.55}, ymode=log, xmode=log]
                \addplot[brown, mark=square*] table[x=MeshSize, y=L2L2lsetTrans]{\ErrKiteTwoDOrderTwo};
                \addlegendentry{\small $e_{L^2}$}
                \addplot[blue, mark=diamond*] table[x=MeshSize, y=LinfLinfExLsetOnGh]{\ErrKiteTwoDOrderTwo};
                \addlegendentry{\small $e_\G^\infty$}
                \addplot[red, mark=*] table[x=MeshSize, y=L2ExLsetOnGh]{\ErrKiteTwoDOrderTwo};
                \addlegendentry{\small $e_\G$}
                \addplot[dashed,line width=0.75pt]{0.4*x^3};
                \addlegendentry{\small $\cO(h^3)$}
                % \addplot[red, mark=diamond*] table[x=MeshSize, y=L2L2lsetCut]{\ErrKiteTwoDOrderTwo};
                % \addlegendentry{\small $e_{L^2}^{\text{Cut}}$}
            \end{axis}
        \end{tikzpicture}
        \caption{BDF$3$, $k=2$}
    \end{subfigure}
    \caption{Error metrics for  kite deforming to a circle in 2D} \label{fig_Kite2DConv}
\end{figure}

\newcommand{\ErrKiteThreeDOrderOne}{./errors/Transport3D/Kite_Order_1_SmallProj_errors.dat}
\newcommand{\ErrKiteThreeDOrderTwo}{./errors/Transport3D/Kite_Order_2_SmallProj_errors.dat}

\begin{figure}[ht!]
    \centering
    \begin{subfigure}[b]{0.40\textwidth}
        \begin{tikzpicture}
            \pgfplotsset{legend style={at={(0.99,0.01)}, anchor=south east, legend columns=1, draw=none, fill=none},}
            \begin{axis}[domain={0.035:0.55}, ymode=log, xmode=log] % ymin=0.00002, ymax=1.25,
                \addplot[brown, mark=square*] table[x=MeshSize, y=L2L2lsetTrans]{\ErrKiteThreeDOrderOne};
                \addlegendentry{\small $e_{L^2}$}
                \addplot[blue, mark=diamond*] table[x=MeshSize, y=LinfLinfExLsetOnGh]{\ErrKiteThreeDOrderOne};
                \addlegendentry{\small $e_\G^\infty$}
                \addplot[red, mark=*] table[x=MeshSize, y=L2ExLsetOnGh]{\ErrKiteThreeDOrderOne};
                \addlegendentry{\small $e_\G$}
                \addplot[black, dashed,line width=0.75pt]{0.6*x^2};
                \addlegendentry{\small $\cO(h^2)$}
                % \addplot[red, mark=diamond*] table[x=MeshSize, y=L2L2lsetCut]{\ErrKiteThreeDOrderOne};
                % \addlegendentry{\small $e_{L^2}^{\text{Cut}}$}
            \end{axis}
        \end{tikzpicture}
        \caption{BDF$2$, $k=1$}
    \end{subfigure}
    \begin{subfigure}[b]{0.40\textwidth}
        \begin{tikzpicture}
            \pgfplotsset{legend style={at={(0.99,0.01)}, anchor=south east, legend columns=1, draw=none, fill=none},}
            \begin{axis}[domain={0.07:0.55}, ymode=log, xmode=log]
                \addplot[brown, mark=square*] table[x=MeshSize, y=L2L2lsetTrans]{\ErrKiteThreeDOrderTwo};
                \addlegendentry{\small $e_{L^2}$}
                \addplot[blue, mark=diamond*] table[x=MeshSize, y=LinfLinfExLsetOnGh]{\ErrKiteThreeDOrderTwo};
                \addlegendentry{\small $e_\G^\infty$}
                \addplot[red, mark=*] table[x=MeshSize, y=L2ExLsetOnGh]{\ErrKiteThreeDOrderTwo};
                \addlegendentry{\small $e_\G$}
                \addplot[dashed,line width=0.75pt]{0.2*x^3};
                \addlegendentry{\small $\cO(h^3)$}
                % \addplot[red, mark=diamond*] table[x=MeshSize, y=L2L2lsetCut]{\ErrKiteThreeDOrderTwo};
                % \addlegendentry{\small $e_{L^2}^{\text{Cut}}$}
            \end{axis}
        \end{tikzpicture}
        \caption{BDF$3$, $k=2$}
    \end{subfigure}
    \caption{Error metrics for kite deforming to a sphere in 3D} \label{fig_Kite3DConv}
\end{figure}

We examine the effect of the inflow boundary data, cf.  Section \ref{s:bc}. We consider kite($2D$) with BDF2 and $k=1$, and compare the results for the first order (in time) boundary data \eqref{bnd_data_BDF1_lo}, the second order boundary values \eqref{bnd_data_BDF2_lo} and a fourth order version  of the boundary data from \eqref{bnd_data_BDF1_ho}. We determined the surface error ($e_\G$) and the $L^2-$error in the narrow band ($e_{L^2}$) for the three choices (low, medium and high accuracy) for boundary data. The results are shown in Figure \ref{fig_BndTest}. We observe almost no difference in the errors after the first mesh refinement. The results demonstrate that due to the choice of a suitable (sufficiently small) projection domain in the extension method we obtain higher order accuracy in the surface approximation even with low order accurate inflow boundary data.

\newcommand{\ErrKiteLoBnd}{./errors/Transport2D/Kite_SpecialBnd/LoBnd_SmallProj_errors.dat}
\newcommand{\ErrKiteHoBnd}{./errors/Transport2D/Kite_SpecialBnd/HoBnd_SmallProj_errors.dat}
\newcommand{\ErrKiteMedBnd}{./errors/Transport2D/Kite_SpecialBnd/MedBnd_SmallProj_errors.dat}

\begin{figure}[ht!]
    \centering
    \begin{subfigure}[b]{0.40\textwidth}
        \begin{tikzpicture}
            \pgfplotsset{legend style={at={(1.55,0.01)}, anchor=south east, legend columns=1, draw=none, fill=none},}
            \begin{axis}[domain={0.0156:0.55}, ymode=log, xmode=log] % ymin=0.00002, ymax=1.25,
                \addplot[brown, dashed, mark=square*] table[x=MeshSize, y=LinfLinfExLsetOnGh]{\ErrKiteLoBnd};
                \addlegendentry{\small $e_{L^2} - \text{low}$}
                \addplot[red, dashed, mark=square*] table[x=MeshSize, y=LinfLinfExLsetOnGh]{\ErrKiteMedBnd};
                \addlegendentry{\small $e_{L^2} - \text{med.}$}
                \addplot[blue, dashed, mark=square*] table[x=MeshSize, y=LinfLinfExLsetOnGh]{\ErrKiteHoBnd};
                \addlegendentry{\small $e_{L^2} - \text{high}$}
                \addplot[brown, mark=*] table[x=MeshSize, y=L2ExLsetOnGh]{\ErrKiteLoBnd};
                \addlegendentry{\small $e_\G - \text{low}$}
                \addplot[red, mark=*] table[x=MeshSize, y=L2ExLsetOnGh]{\ErrKiteMedBnd};
                \addlegendentry{\small $e_\G - \text{med.}$}
                \addplot[blue, mark=*] table[x=MeshSize, y=L2ExLsetOnGh]{\ErrKiteHoBnd};
                \addlegendentry{\small $e_\G - \text{high}$}
            \end{axis}
        \end{tikzpicture}
        \end{subfigure}
    \caption{Kite deforming to a circle in 2D, BDF2, $k=1$; different inflow boundary data} \label{fig_BndTest}
\end{figure}
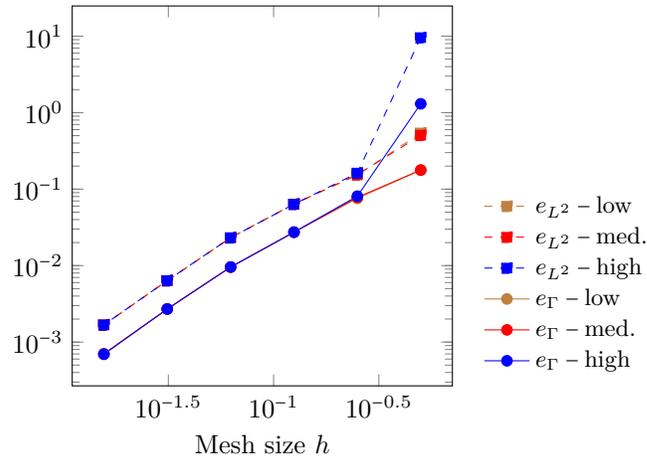

\subsection{Ball in a rotating flow field}
The exact level set function is given by $\phi(x, t) = (x_1 - \cos(2\pi t))^2 + (x_2 - \sin(2\pi t))^2 - \frac{1}{2}$ in 2D and $\phi(x, t) = (x_1 - \cos(2\pi t))^2 + (x_2 - \sin(2\pi t))^2 + x_3^2 - \frac{1}{2}$ in 3D. This level set function describes a sphere (or circle) that makes a full rotation around the origin for $t\in[0,1]$. The velocity field of the rotation is given by $\bu = 2\pi (-x_2, x_1)^T$ in 2D and $\bu = 2\pi (-x_2, x_1, 0)^T$ in 3D. The geometry evolution  is shown in Figure \ref{fig_TravelingSphere}.

\begin{figure}[ht!]
    \centering
    \begin{subfigure}[b]{0.35\textwidth}
        \centering
        \includegraphics[width=0.8\textwidth]{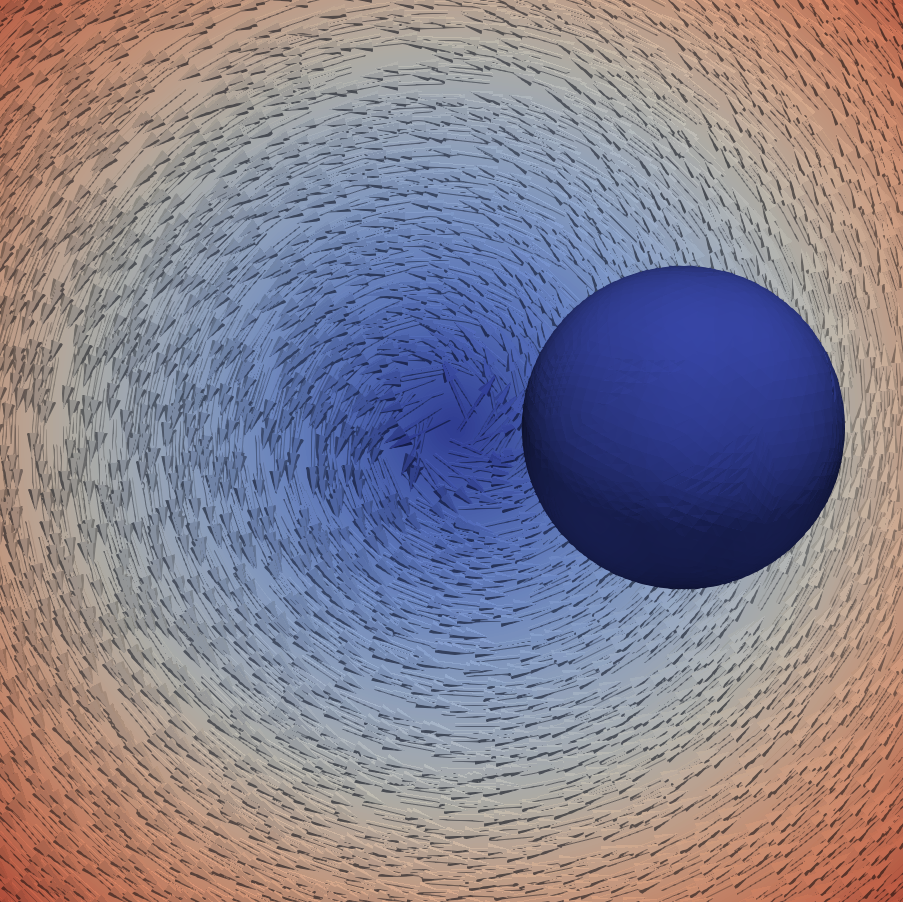}
        \caption{$t = 0$}
    \end{subfigure}
    \hspace*{0.8cm}
    \begin{subfigure}[b]{0.35\textwidth}
        \centering
        \includegraphics[width=0.8\textwidth]{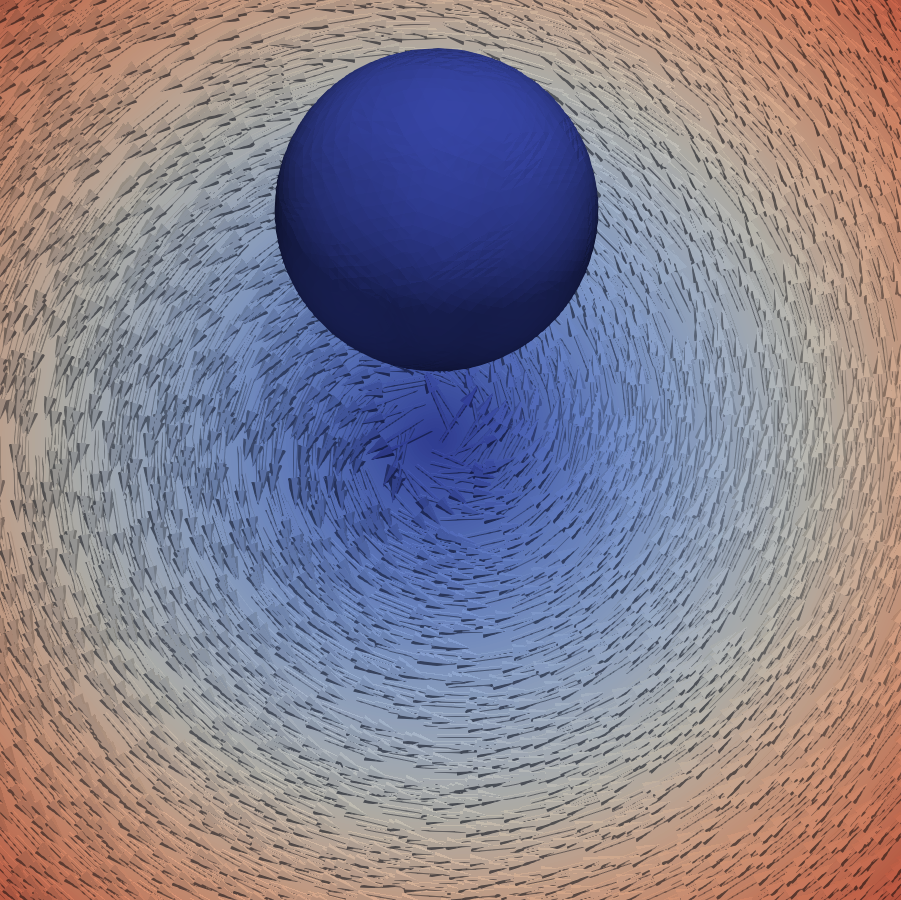}
        \caption{$t = 0.25$}
    \end{subfigure}
    \begin{subfigure}[b]{0.35\textwidth}
        \centering
        \includegraphics[width=0.8\textwidth]{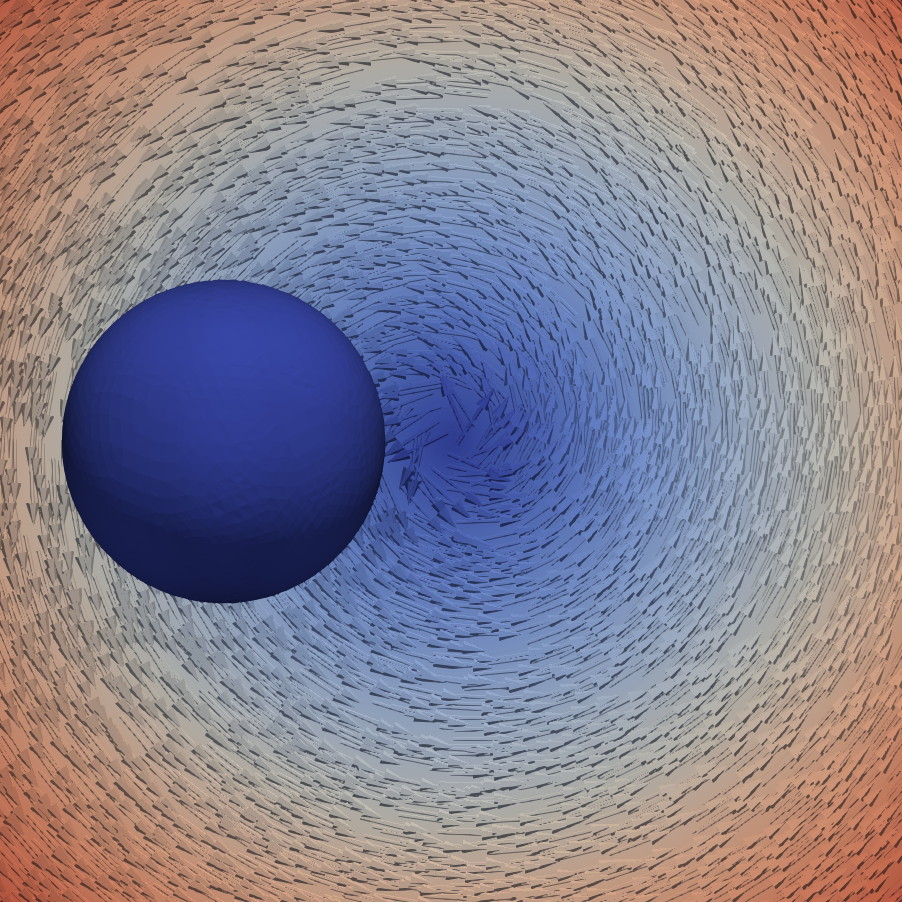}
        \caption{$t = 0.5$}
        \end{subfigure}
    \hspace*{0.8cm}
    \begin{subfigure}[b]{0.35\textwidth}
        \centering
        \includegraphics[width=0.8\textwidth]{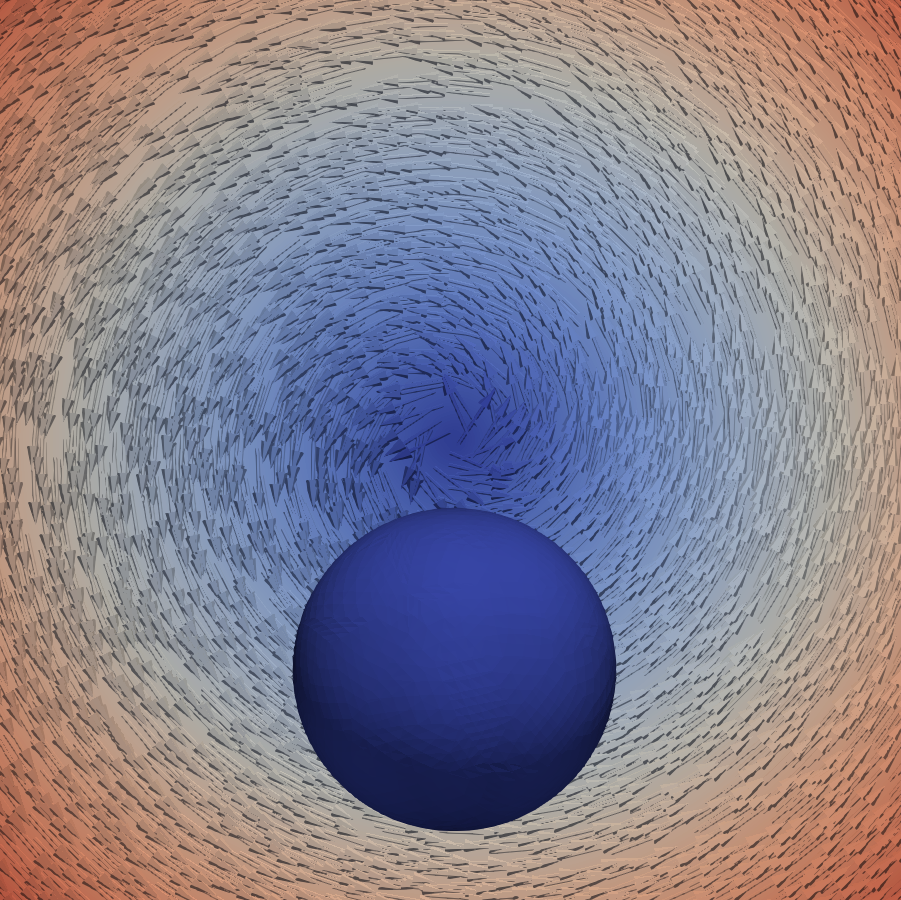}
        \caption{$t = 0.75$}
    \end{subfigure}
    \caption{Sphere in rotating flow field}\label{fig_TravelingSphere}
\end{figure}

The coarsest mesh size is  $h_0 = 0.25$, and we halve it in each refinement step. The results are shown in Fig.~\ref{fig_Sphere2DConv} for the 2D case and in Fig.~\ref{fig_Sphere3DConv} for the 3D case. We observe a deterioration of the convergence rate in the last refinement in the $L^2-$error in the narrow band in two dimensions for $k=2$. Recall, however, that our main interest is in an accurate surface approximation.   For this we observe optimal convergence rates of $\cO(h^2)$ for $k=1$ and $\cO(h^3)$ for $k=2$ with all corresponding error measures in two and three dimensions.

\newcommand{\ErrSphereTwoDOrderOne}{./errors/Transport2D/Sphere_Order_1_SmallProj_errors.dat}
\newcommand{\ErrSphereTwoDOrderTwo}{./errors/Transport2D/Sphere_Order_2_SmallProj_errors.dat}

\begin{figure}[ht!]
    \centering
    \begin{subfigure}[b]{0.40\textwidth}
        \begin{tikzpicture}
            \pgfplotsset{legend style={at={(0.99,0.01)}, anchor=south east, legend columns=1, draw=none, fill=none},}
            \begin{axis}[domain={0.0156:0.28}, ymode=log, xmode=log] % ymin=0.00002, ymax=1.25,
                \addplot[brown, mark=square*] table[x=MeshSize, y=L2L2lsetTrans]{\ErrSphereTwoDOrderOne};
                \addlegendentry{\small $e_{L^2}$}
                \addplot[blue, mark=diamond*] table[x=MeshSize, y=LinfLinfExLsetOnGh]{\ErrSphereTwoDOrderOne};
                \addlegendentry{\small $e_\G^\infty$}
                \addplot[red, mark=*] table[x=MeshSize, y=L2ExLsetOnGh]{\ErrSphereTwoDOrderOne};
                \addlegendentry{\small $e_\G$}
                \addplot[black, dashed,line width=0.75pt]{0.6*x^2};
                \addlegendentry{\small $\cO(h^2)$}
                % \addplot[red, mark=diamond*] table[x=MeshSize, y=L2L2lsetCut]{\ErrSphereTwoDOrderOne};
                % \addlegendentry{\small $e_{L^2}^{\text{Cut}}$}
            \end{axis}
        \end{tikzpicture}
        \caption{BDF$2$, $k=1$}
    \end{subfigure}
    \begin{subfigure}[b]{0.40\textwidth}
        \begin{tikzpicture}
            \pgfplotsset{legend style={at={(0.99,0.01)}, anchor=south east, legend columns=1, draw=none, fill=none},}
            \begin{axis}[domain={0.0156:0.28}, ymode=log, xmode=log]
                \addplot[brown, mark=square*] table[x=MeshSize, y=L2L2lsetTrans]{\ErrSphereTwoDOrderTwo};
                \addlegendentry{\small $e_{L^2}$}
                \addplot[blue, mark=diamond*] table[x=MeshSize, y=LinfLinfExLsetOnGh]{\ErrSphereTwoDOrderTwo};
                \addlegendentry{\small $e_\G^\infty$}
                \addplot[red, mark=*] table[x=MeshSize, y=L2ExLsetOnGh]{\ErrSphereTwoDOrderTwo};
                \addlegendentry{\small $e_\G$}
                \addplot[dashed,line width=0.75pt]{0.02*x^3};
                \addlegendentry{\small $\cO(h^3)$}
                % \addplot[red, mark=diamond*] table[x=MeshSize, y=L2L2lsetCut]{\ErrSphereTwoDOrderTwo};
                % \addlegendentry{\small $e_{L^2}^{\text{Cut}}$}
            \end{axis}
        \end{tikzpicture}
        \caption{BDF$3$, $k=2$}
    \end{subfigure}
    \caption{Error metrics for circle rotating around the origin in 2D} \label{fig_Sphere2DConv}
\end{figure}
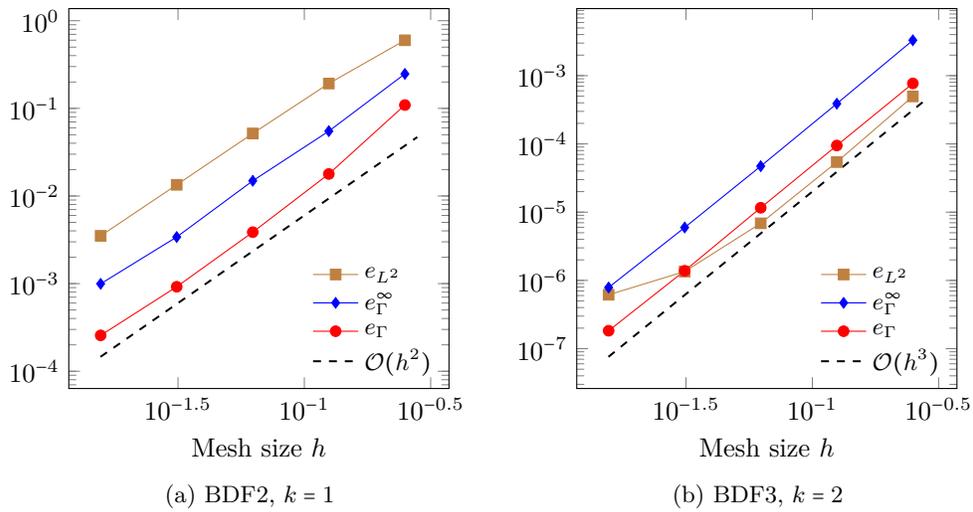

\newcommand{\ErrSphereThreeDOrderOne}{./errors/Transport3D/Sphere_Order_1_SmallProj_errors.dat}
\newcommand{\ErrSphereThreeDOrderTwo}{./errors/Transport3D/Sphere_Order_2_SmallProj_errors.dat}

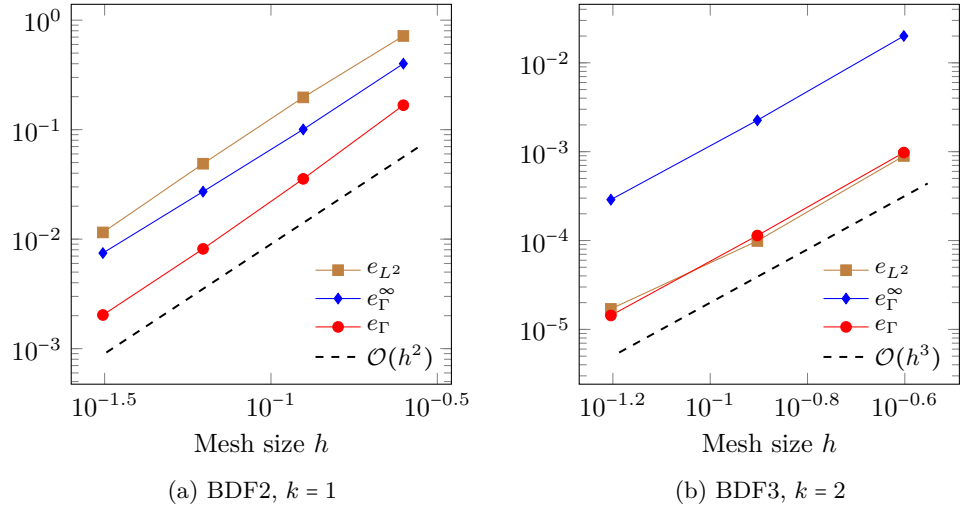
\begin{figure}[ht!]
    \centering
    \begin{subfigure}[b]{0.40\textwidth}
        \begin{tikzpicture}
            \pgfplotsset{legend style={at={(0.99,0.01)}, anchor=south east, legend columns=1, draw=none, fill=none},}
            \begin{axis}[domain={0.032:0.28}, ymode=log, xmode=log] % ymin=0.00002, ymax=1.25,
                \addplot[brown, mark=square*] table[x=MeshSize, y=L2L2lsetTrans]{\ErrSphereThreeDOrderOne};
                \addlegendentry{\small $e_{L^2}$}
                \addplot[blue, mark=diamond*] table[x=MeshSize, y=LinfLinfExLsetOnGh]{\ErrSphereThreeDOrderOne};
                \addlegendentry{\small $e_\G^\infty$}
                \addplot[red, mark=*] table[x=MeshSize, y=L2ExLsetOnGh]{\ErrSphereThreeDOrderOne};
                \addlegendentry{\small $e_\G$}
                \addplot[black, dashed,line width=0.75pt]{0.9*x^2};
                \addlegendentry{\small $\cO(h^2)$}
                % \addplot[red, mark=diamond*] table[x=MeshSize, y=L2L2lsetCut]{\ErrSphereThreeDOrderOne};
                % \addlegendentry{\small $e_{L^2}^{\text{Cut}}$}
            \end{axis}
        \end{tikzpicture}
        \caption{BDF$2$, $k=1$}
    \end{subfigure}
    \begin{subfigure}[b]{0.40\textwidth}
        \begin{tikzpicture}
            \pgfplotsset{legend style={at={(0.99,0.01)}, anchor=south east, legend columns=1, draw=none, fill=none},}
            \begin{axis}[domain={0.065:0.28}, ymode=log, xmode=log]
                \addplot[brown, mark=square*] table[x=MeshSize, y=L2L2lsetTrans]{\ErrSphereThreeDOrderTwo};
                \addlegendentry{\small $e_{L^2}$}
                \addplot[blue, mark=diamond*] table[x=MeshSize, y=LinfLinfExLsetOnGh]{\ErrSphereThreeDOrderTwo};
                \addlegendentry{\small $e_\G^\infty$}
                \addplot[red, mark=*] table[x=MeshSize, y=L2ExLsetOnGh]{\ErrSphereThreeDOrderTwo};
                \addlegendentry{\small $e_\G$}
                \addplot[dashed,line width=0.75pt]{0.02*x^3};
                \addlegendentry{\small $\cO(h^3)$}
                % \addplot[red, mark=diamond*] table[x=MeshSize, y=L2L2lsetCut]{\ErrSphereThreeDOrderTwo};
                % \addlegendentry{\small $e_{L^2}^{\text{Cut}}$}
            \end{axis}
        \end{tikzpicture}
        \caption{BDF$3$, $k=2$}
    \end{subfigure}
    \caption{Error metrics for sphere rotating around the origin in 3D} \label{fig_Sphere3DConv}
\end{figure}

In the next experiment, we investigate how a ``too large'' projection domain in the ghost penalty extension method affects the accuracy of the discretization method. We compare the projection domain $\Omega_h^p = \cN^1(\cT_\G)$, which is used in all experiments above, with the maximal possible projection domain $\Omega_h^p = \Omega_\G^n \cap \Omega_\G^{n+1}$. We use first order polynomials in space and a BDF2 method in time. The surface errors ($L^2$- and $L^\infty$-errors) for the 2D case are shown in Figure \ref{fig_Sphere2D_diff_Proj}. We observe that with the maximal projection domain we obtain a suboptimal  rate of convergence. This confirms the heuristics discussed in Section~\ref{sec_transport_whole_algo}.
%This supports our choice of the projection domain in the extension method. We believe that this behavior arises because the error in the inflow boundary values are transported inside of the transport domain and the extension with a small projection domain does neglect these errors, while the choice of the maximal projection domain does not.

\newcommand{\ErrSphereTwoDOrderOneMaxProj}{./errors/Transport2D/Sphere_Order_1_MaxProj_errors.dat}

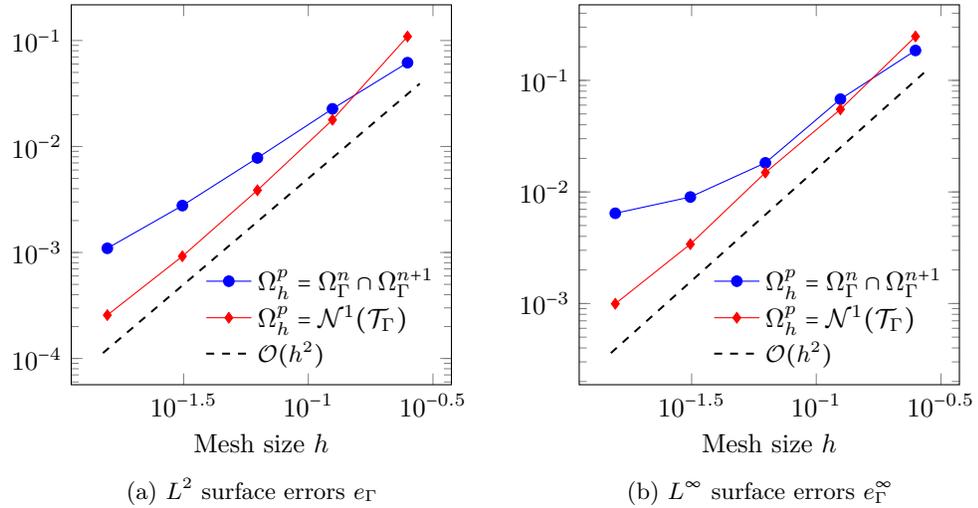
\begin{figure}[ht!]
    \centering
    \begin{subfigure}[b]{0.40\textwidth}
        \begin{tikzpicture}
            \pgfplotsset{legend style={at={(0.99,0.02)}, anchor=south east, legend columns=1, draw=none, fill=none},}
            \begin{axis}[domain={0.015:0.28}, ymode=log, xmode=log]
                \addplot[blue, mark=*] table[x=MeshSize, y=L2ExLsetOnGh]{\ErrSphereTwoDOrderOneMaxProj};
                \addlegendentry{$\Omega_h^p = \Omega_\G^n \cap \Omega_\G^{n+1}$}
                \addplot[red, mark=diamond*] table[x=MeshSize, y=L2ExLsetOnGh]{\ErrSphereTwoDOrderOne};
                \addlegendentry{$\Omega_h^p = \cN^1(\cT_\G)$}
                \addplot[dashed,line width=0.75pt]{0.5*x^2};
                \addlegendentry{\small $\cO(h^2)$}
            \end{axis}
        \end{tikzpicture}
        \caption{$L^2$ surface errors $e_\G$}
    \end{subfigure}
    \begin{subfigure}[b]{0.40\textwidth}
        \begin{tikzpicture}
            \pgfplotsset{legend style={at={(0.99,0.02)}, anchor=south east, legend columns=1, draw=none, fill=none},}
            \begin{axis}[domain={0.015:0.28}, ymode=log, xmode=log]
                \addplot[blue, mark=*] table[x=MeshSize, y=LinfLinfExLsetOnGh]{\ErrSphereTwoDOrderOneMaxProj};
                \addlegendentry{$\Omega_h^p = \Omega_\G^n \cap \Omega_\G^{n+1}$}
                \addplot[red, mark=diamond*] table[x=MeshSize, y=LinfLinfExLsetOnGh]{\ErrSphereTwoDOrderOne};
                \addlegendentry{$\Omega_h^p = \cN^1(\cT_\G)$}
                \addplot[dashed,line width=0.75pt]{1.6*x^2};
                \addlegendentry{\small $\cO(h^2)$}
            \end{axis}
        \end{tikzpicture}
        \caption{$L^\infty$ surface errors $e_\G^\infty$}
    \end{subfigure}
    \caption{Different projection domains in extension method} \label{fig_Sphere2D_diff_Proj}
\end{figure}

In a further experiment, we investigate the volume conservation in 3D. Since the velocity field  is divergence-free, in the continuous problem we have volume conservation. We measure the enclosed volume of the numerical approximation $\Gamma_h(t_n)$ in each time step for the lower and higher order narrow band scheme with fix mesh size $h = 0.0625$. The results are shown in Figure \ref{fig_3DVolumeConv}, and show that using higher order polynomials ($k = 2$), leads to a very accurate  volume conservation.

\newcommand{\ErrSphereThreeDOrderOneVolume}{./errors/Transport3D/VolumeErrors/Sphere_Order_1_SmallProj_Lx3_VolumeErrors.dat}
\newcommand{\ErrSphereThreeDOrderTwoVolume}{./errors/Transport3D/VolumeErrors/Sphere_Order_2_SmallProj_Lx3_VolumeErrors.dat}

\begin{figure}[ht!]
    \centering
    \begin{subfigure}[b]{0.40\textwidth}
        \begin{tikzpicture}
            \pgfplotsset{legend style={at={(0.6,0.02)}}, xlabel={Time $t$}, ylabel={error in enclosed volume}}
            \begin{axis}[no markers, ymode=log, every axis plot/.append style={ultra thick}] % ymin=1.44, ymax=1.49,
                \addplot[red, mark=] table[x=Time, y=Error]{\ErrSphereThreeDOrderOneVolume};
                \addlegendentry{\small BDF$2$, $k=1$}
                \addplot[blue, mark=] table[x=Time, y=Error]{\ErrSphereThreeDOrderTwoVolume};
                \addlegendentry{\small BDF$3$, $k=2$}
            \end{axis}
        \end{tikzpicture}
    \end{subfigure}
    \caption{Rotating sphere; error in enclosed volume} \label{fig_3DVolumeConv}
\end{figure}
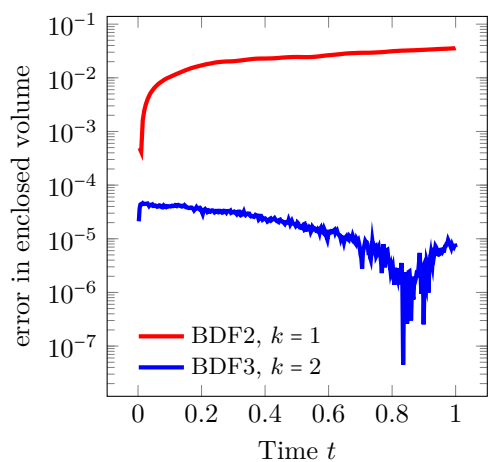

\subsection{Strongly deforming sphere.}

As an initial surface we take a sphere with radius $0.15$ centered at $(0.5, 0.75, 0.5)$. A corresponding level set function is given by $\phi(x,y,z) = (x-0.5)^2 + (y-0.75)^2 + (z-0.5)^2 - 0.15^2$. The sphere is transported and deformed in a velocity field given by
\begin{equation*}
    \bu = \cos\left(\frac{\pi}{T} t\right)
    \left(\begin{array}{c}
        -2 \sin(\pi x)^2 \sin(\pi y) \cos(\pi y)\\
         2 \sin(\pi y)^2 \sin(\pi x) \cos(\pi x)\\
         0
    \end{array}\right), \quad t \in [0,T].
\end{equation*}
Due to the periodic $t$-behaviour of this velocity field  the surface will deform until $t = T/2$ and then return to the initial spherical shape at $t = T$.
The surface exhibits a spiral-like deformation. Increasing $T$ results in  stronger deformed shapes of \Gt at $t = T/2$.  In Figure \ref{fig_DeformingDrop} we show resulting deformations of the  sphere for $T = 2$ and $T = 4$. The maximal deformation can be observed at $t = T/2$. Similar test cases are widely used to test interface capturing and interface tracking methods, cf. \cite{DiPietroLoForteParolini2006, EnrightFedkiwFerzigerMitchell2002, MarchandiseRemacleChevaugeon2006}.
\begin{figure}[ht!]
    \centering
    \begin{subfigure}[b]{0.23\textwidth}
        \centering
        \includegraphics[width=\textwidth]{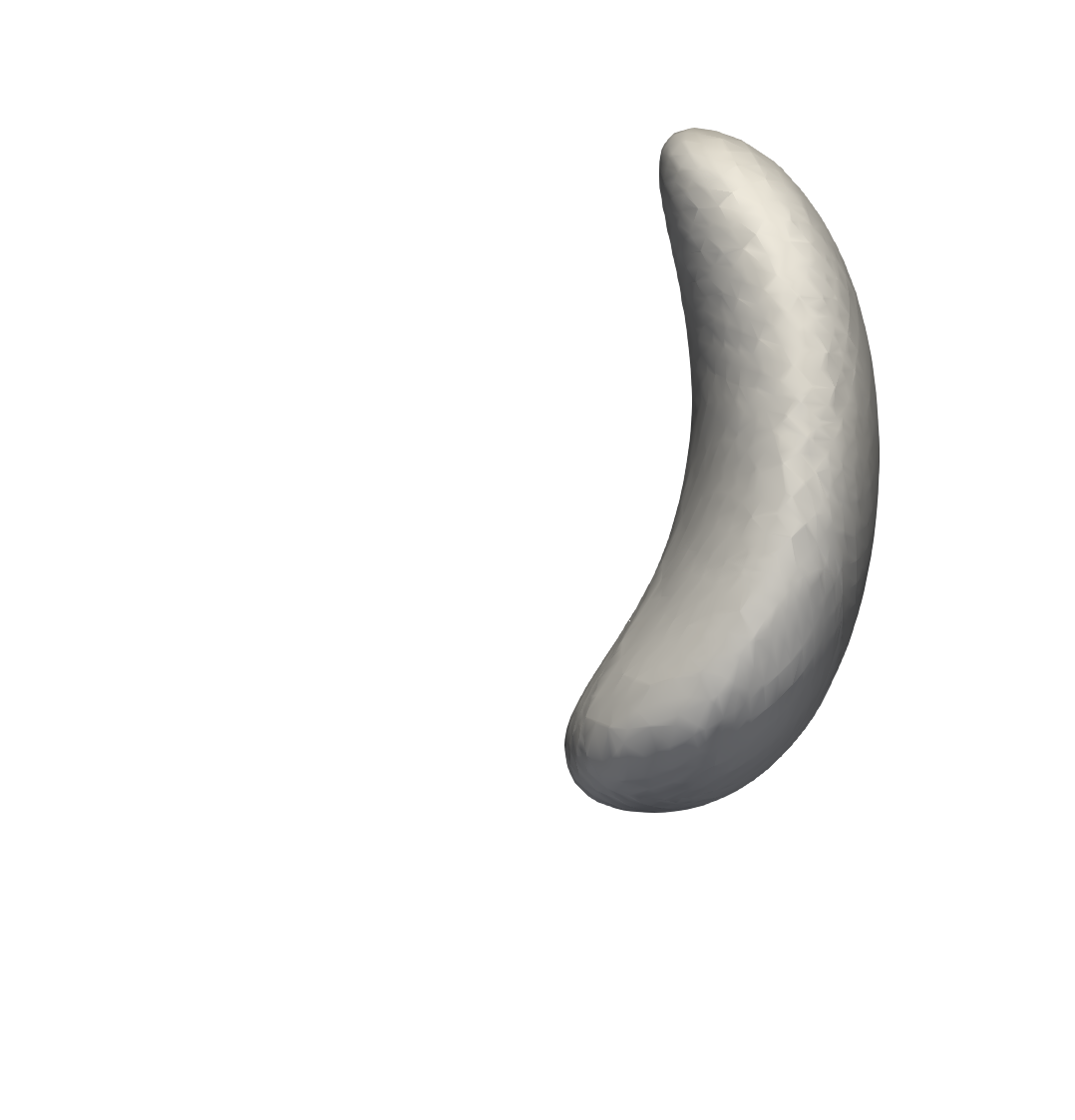}
        \caption{$T = 2$, $t = \frac{T}{4}$}
    \end{subfigure}
    \begin{subfigure}[b]{0.23\textwidth}
        \centering
        \includegraphics[width=\textwidth]{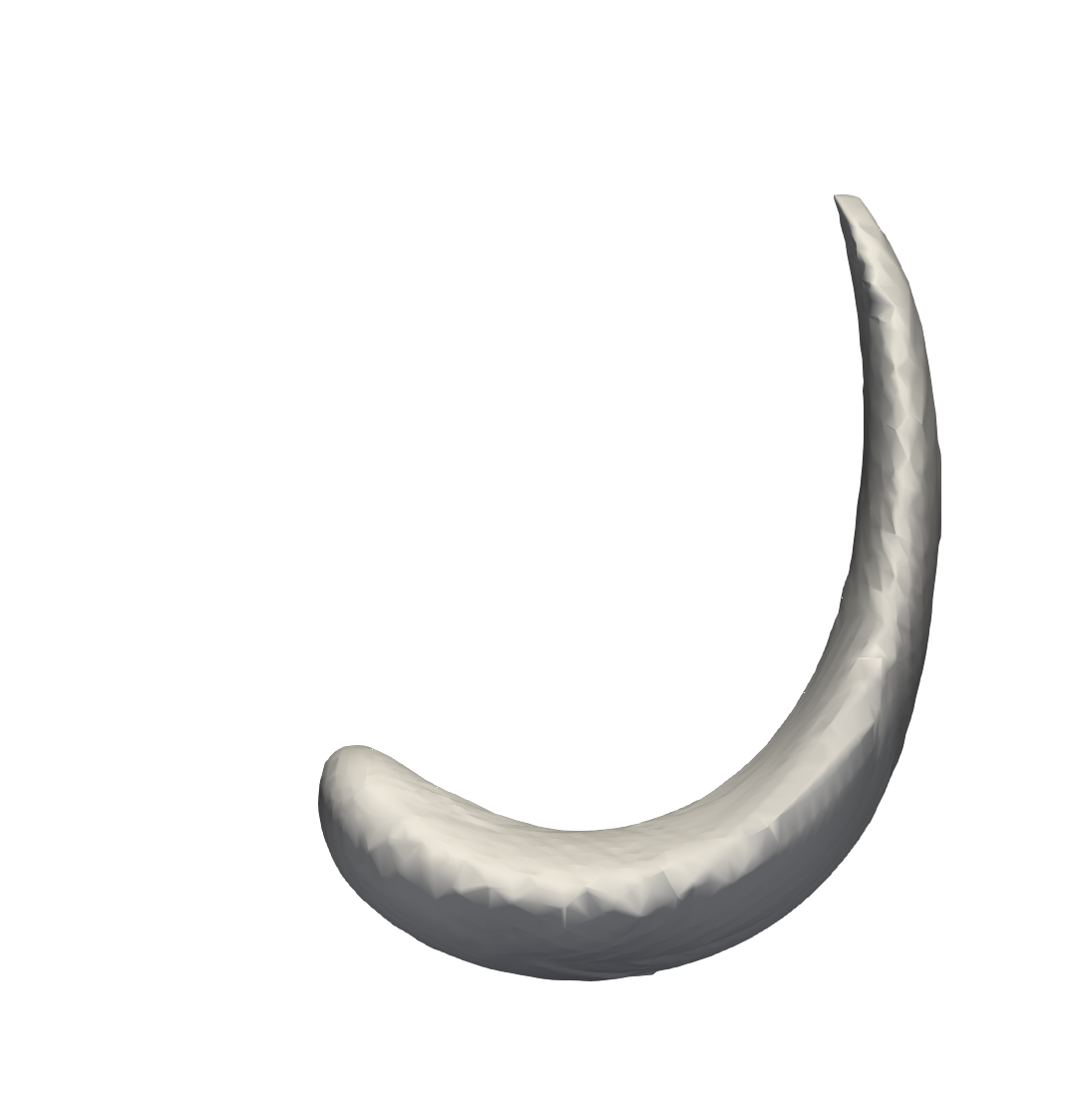}
        \caption{$T = 2$, $t = \frac{T}{2}$}
    \end{subfigure}
    \begin{subfigure}[b]{0.23\textwidth}
        \centering
        \includegraphics[width=\textwidth]{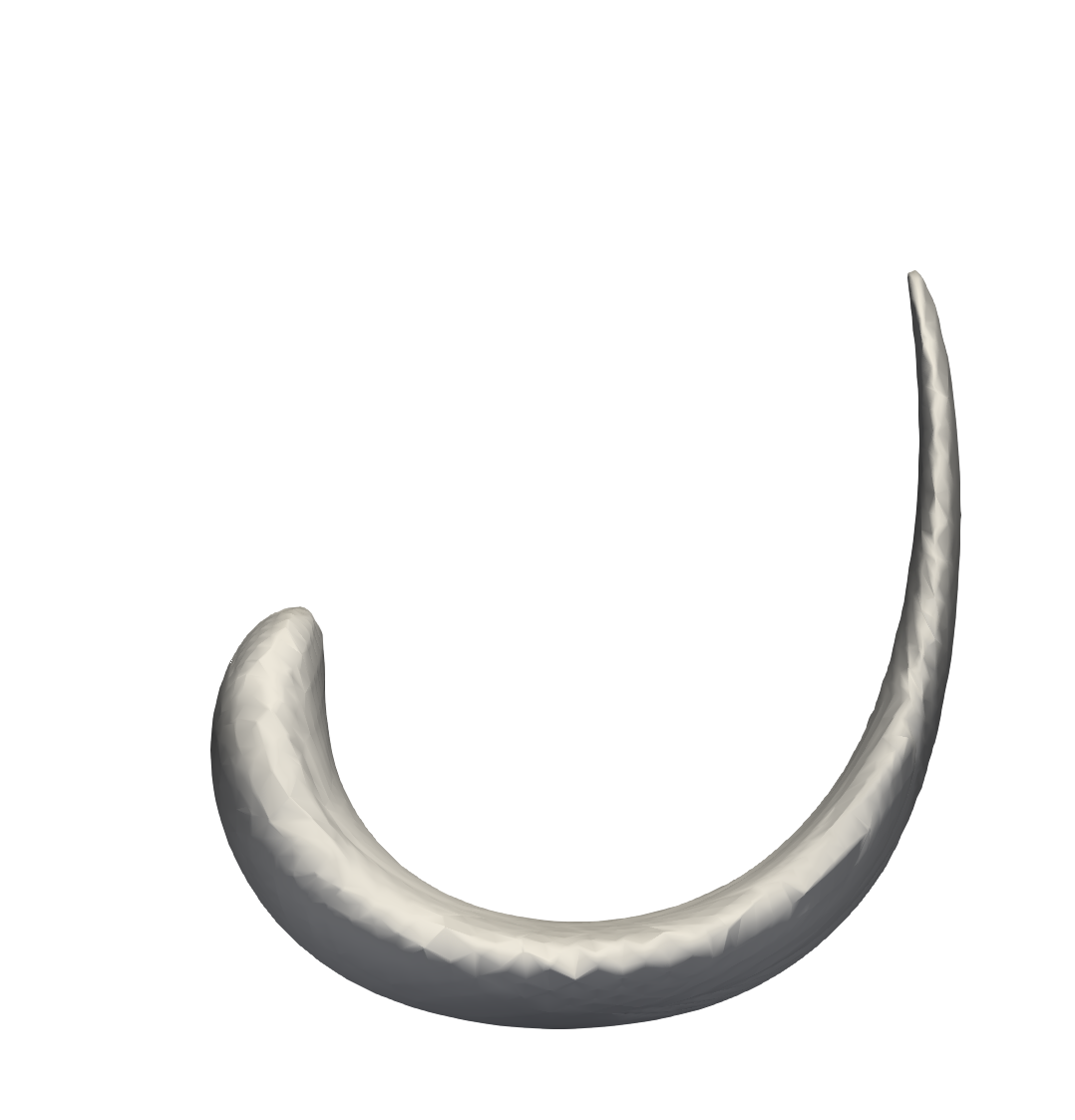}
        \caption{$T = 4$, $t = \frac{T}{4}$}
    \end{subfigure}
    \begin{subfigure}[b]{0.23\textwidth}
        \centering
        \includegraphics[width=\textwidth]{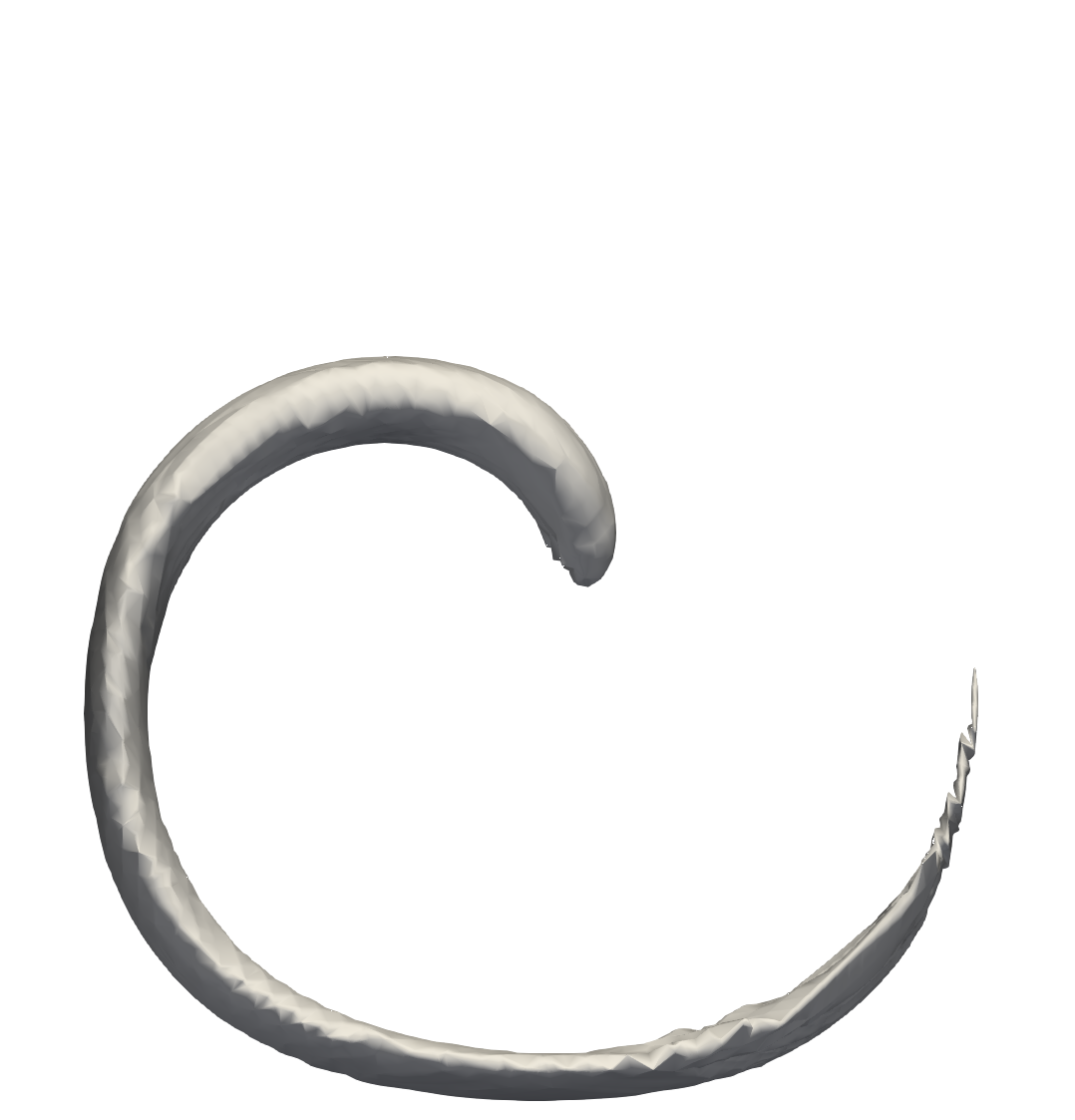}
        \caption{$T = 4$, $t = \frac{T}{2}$}
    \end{subfigure}
    \caption{Deformations of the sphere in a vortex} \label{fig_DeformingDrop}
\end{figure}

We define the computational domain as $\Omega = [0,1]^3$ and apply the transport scheme on the cut elements augmented by four layers of neighboring elements, denoted as $\Omega_\Gamma^n := \mathcal{N}^4(\mathcal{T}_\Gamma^n)$. The projection domain in the extension method consists of the cut elements with two layers of neighboring elements added, i.e. $\Omega_h^{p, n+1} = \mathcal{N}^2(\mathcal{T}_\Gamma^{n+1})$. We use the BDF2 and BDF3 schemes combined with finite elements of degree $1$ and $2$, respectively. In the previous examples, we took a variable time step size $\Delta t \sim h / \|V_\G^n \|_\infty$ to ensure that the zero level of the surface remains in the transport domain $\Omega_\Gamma^n$ within the time step. However, this time step size criterion is not satisfactory in this example because the velocity $V_\G$ vanishes at $t = T/2$. Instead of a variable time step size,  we take a constant time step size $\Delta t = h$ in the following experiments.
%Since the surface velocity does not increase rapidly, this choice of the time step size is reasonable.
At the end of each time step, we check if the zero level of the level set function stays in the transport domain. If this is not the case, the algorithm breaks down.

In this example the exact level set function is not known for $t \notin \{0,T\}$. As a measure of accuracy we use
\[
|e_{\G,N}|^2 = \oint_{\Gh(T)} \phi(\bx,T)^2 \dx,
\]
 which quantifies the deviation of the discrete zero level set from the spherical shape at the final time $t=T$.

We present results of our method applied to this deforming sphere example for $T=2$ in Figure~\ref{fig_DefDrop_T2}. The initial mesh and time step sizes are taken as $h_0 = \Delta t_0 = 2^{-4}=0.0625$ and are halved in each refinement step. We observe a convergence rate of (approximately) $\cO(h^{1.5})$ for $k=1$ and second order convergence for $k=2$.

\begin{figure}[ht!]
    \centering
    \begin{subfigure}[b]{0.50\textwidth}
        \begin{tikzpicture}
            \pgfplotsset{legend style={at={(0.99,0.02)}, anchor=south east, legend columns=1, draw=none, fill=none}}
            \begin{axis}[domain={0.0078:0.0625}, ymode=log, xmode=log]
                \addplot[red, mark=o] table [x=h, y=L2ExLsetOnGh, col sep=comma]{./T2_k1.csv}; \addlegendentry{$k=1$}
                \addplot[blue, mark=*] table [x=h, y=L2ExLsetOnGh, col sep=comma, skip coords between index={1}{3}]{./T2_k2.csv}; \addlegendentry{$k=2$}
                \addplot[dotted,line width=0.75pt]{2*x^1.5}; \addlegendentry{\small $\cO(h^{1.5})$}
                \addplot[dashed,line width=0.75pt]{2*x^2}; \addlegendentry{\small $\cO(h^2)$}
            \end{axis}
        \end{tikzpicture}
    \end{subfigure}
    \caption{$e_{\G, N}$ for $T = 2$} \label{fig_DefDrop_T2}
\end{figure}

We show the values of $e_{\G,N}$ for different end times $T$ in Table \ref{table_DefDrop3D}. For certain combinations of end times $T$ and refinement levels, we cannot report errors due to insufficient mesh resolution. By taking sufficiently small $h$ and $k = 1$, we can handle the case $T=4$ in which a very strong deformation occurs. For $T = 3$ and $T = 4$ we do not have results for $k = 2$ since the zero level of the level set function leaves the transport domain $\Omega_\Gamma^n$ for the given time step size $\Delta t = h$ and the algorithm breaks down. As expected, the errors are smaller for smaller $T$ values.  The convergence rates in the table vary between $1$ and $2$ for $k=1$ and between $1.5$ and $2.5$ for $k=2$.

\begin{figure}[ht!]
    \centering
    \includegraphics[width=\textwidth]{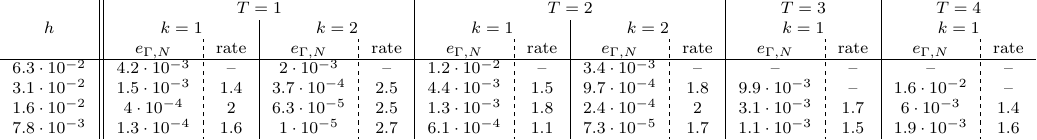}
    \caption{$e_{\G,N}$ errors at final time $T$.}
    \label{table_DefDrop3D}
\end{figure}

%In our experiments we do not observe optimal convergence rates for the surface error. The theory in \cite{Feistauer2016} only proofs suboptimal convergence rates for the transport scheme we use. However, in the previous experiments we observed optimal convergence rates which indicates, that the error bounds in \cite{Feistauer2016} are not sharp.
The suboptimal rates in this experiment can be explained by the rather poor mesh resolution. We do not use an adaptive mesh refinement strategy and the mesh resolution is too low to resolve the strong deformations of the sphere. In Figure~\ref{fig_Zoom} we show a zoom of the narrow tail of the deformed sphere for $T=4$, $t \approx T/2$ and $k=1$. The mesh shown is for $h=2^{-6}=0.016$ (second refinement level). The cut elements are marked in red and the transport elements in blue. As one can see from this figure,  we  have a poor resolution of the part of the surface that is most strongly deformed.
%Thus, the mesh is not fine enough to resolve the strongly deformed sphere. The error in the surface approximation is therefore larger than in the previous experiments.
To overcome this problem one could use an adaptive mesh refinement strategy, which is a topic of further research. On the other hand, one may conclude that even with this relatively poor resolution we obtain fairly good results, which indicates that our method has good robustness properties with respect to strong shape deformations.

\begin{figure}[ht!]
    \centering
    \includegraphics[width=0.4\textwidth]{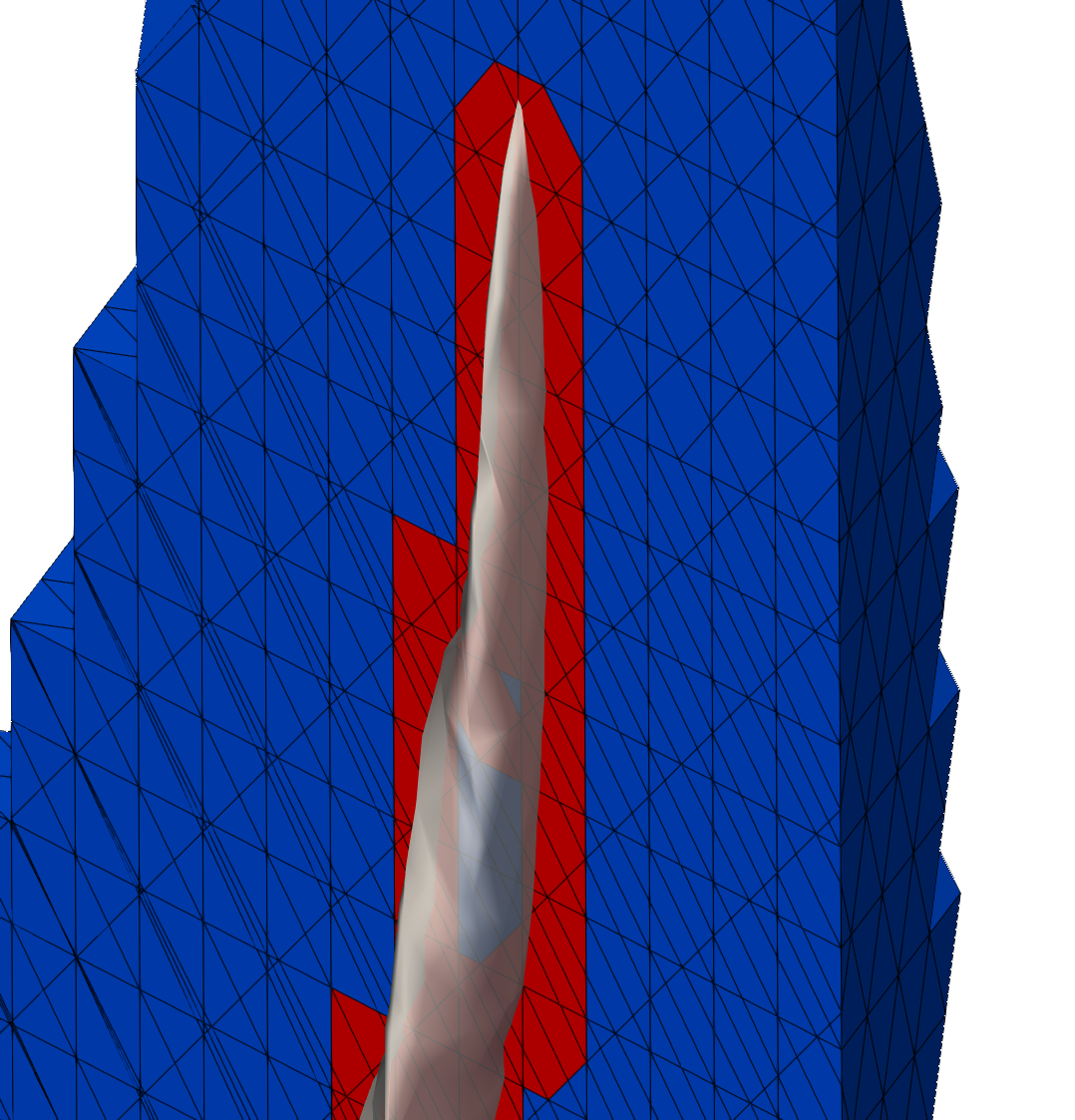}
    \caption{Zoom of the narrow tail of the deformed sphere, $h=2^{-6}=0.016$.} \label{fig_Zoom}
\end{figure}

\subsection{Conclusion}

A narrow-band finite element method with projection-based extension has demonstrated both effectiveness and reliability in capturing the evolution of an interface represented by the zero level of a level set function. The approach is flexible, as higher-order interface approximations can be readily achieved by selecting finite elements of higher degree and employing higher-order time integration techniques. For instance, using piecewise quadratic DG elements together with BDF3 for time integration demonstrated nearly optimal $O(h^3)$ accuracy in recovering the interface position. Interface curvatures and the normal field can be computed through straightforward elementwise calculations.

Exploiting variational properties of the projection, we have proven stability and accuracy estimates for the extension procedure. However, the error analysis of the complete method, which also involves the DG finite element formulation for the transport equation, remains an open problem for future research.

Numerical examples indicate that re-initialization of the level set function is not necessary to ensure the method's stability, even under rather large  deformations of the zero level set. In applications with (very) large deformations  or in problems where an accurate approximation of the distance to $\Gamma_h$ is required, we recommend post-processing $\phi_h^n$ with some standard re-distancing technique (e.g., a variant of fast marching) to obtain $\psi_h^n \simeq \text{dist}(\cdot, \Gamma_h^n)$. Using this post-processing only every $k$th time step, with $k$ ``large'',  avoids systematic errors that could arise from using the re-distancing of $\phi_h^n$ in the extension procedure in every time step.

In applications with large deformations the extension method that we propose  may benefit from  mesh adaptivity. This topic has not been studied, yet.

{\textbf{Acknowledgments:}} The authors A. Reusken and P. Schwering wish to thank the German Research Foundation (DFG) for financial support within the Research Unit ``Vector- and tensor valued surface PDEs'' (FOR 3013) with project no. RE 1461/11-2. The author M. Olshanskii was partially supported by National Science Foundation under Grant No. DMS-2408978.

This material is partially based upon work supported by the National Science Foundation under Grant No. DMS-1929284 while the authors M.O. and A.R. were in residence at the Institute for Computational and Experimental Research in Mathematics in Providence, RI, during the semester program.

% ==============================================================================
\bibliographystyle{siam}
\bibliography{literature}{}
% ==============================================================================
\end{document}

%% file: preamble.tex
\usepackage[utf8]{inputenc}             % um Umlaute direkt eingeben zu können
\usepackage[intlimits]{amsmath} 	    % für erweiterte mathematische Konstrukte
\usepackage{algorithm}
\usepackage{amsfonts,amssymb}    	    % für erweiterte mathematische Schriften und Symbole
\usepackage{array}              	    % für erweiterte Tabellensatzmöglichkeiten
\usepackage[american]{babel}	% für neue Deutsche Rechtschreibung, letzte Sprache ist aktiv
\usepackage{booktabs}                   % Nice rules for tables. Usage \begin{tabular}\toprule ... \midrule ... \bottomrule
\usepackage{cite}
\usepackage{changepage}                 % um Abschnitte einzeln einzurücken
\usepackage{charter}                    % for chapterstyle in KOMA class
\usepackage{chngcntr}				    % Counter ändern
	\counterwithin{figure}{section}	    % damit in Figures die Kapitelnummer mit dran steht
\usepackage{commath}				    % viele praktische Befehle zur Größe von Operatoren
\usepackage{csquotes}				    % für Rechtschreibung in cite??
\usepackage{enumerate}          	    % für Aufzählungen
\usepackage{float}					    % positioning of figures
\usepackage[Glenn]{fncychap}            % macht Probleme mit der Koma Klasse - für die Optik von section: %Options: Sonny, Lenny, Bjornstrup
\usepackage[T1]{fontenc}       		    % für "modernere" Schriftkodierung
\usepackage[pdftex]{graphicx}      	    % für mehr Freiheit bei \includegraphics
\usepackage{IEEEtrantools} 			    % für Mathe Umbgebung
\usepackage{lmodern}				    % Schriftart
\usepackage{latexsym}
\usepackage{mathtools}          	    % useful tools for math mode
\usepackage{MnSymbol}
\usepackage{nicefrac}           	    %
\usepackage{pifont}
\usepackage{relsize}                    % for mathlarger command
\usepackage[markcase=noupper,draft=false]{scrlayer-scrpage}   % remove the uppercasing
\usepackage{scrhack}                    % to get rid of a warning
\usepackage{subcaption}                 % for subfigure enviroment
\usepackage{textcomp}           	    % für erweiterten Text-Symbolvorrat
\usepackage{titlesec}           	    % Kontrolle der Abschnittüberschriften - Probleme mit KOMA script
\usepackage{twoopt}                     % für commands mit zwei DEFAULT Werten
\usepackage[usenames,dvipsnames,xcdraw,table]{xcolor}
\usepackage{xargs}					    % für commands mit zwei DEFAULT Werten
\usepackage{xspace}                     % für space nach ensuremath, der nicht kommt, wenn ein Punkt o.Ä. im Text steht

\usepackage{eso-pic}
\usepackage{pgfplots}
\usepackage{pgfplotstable}
    \pgfplotsset{width=\linewidth, height=\linewidth, xlabel={Mesh size $h$}, legend cell align=left, cycle list name=mark list}
    \pgfplotsset{compat=1.11}
    \pgfplotsset{legend style={at={(.99,0.1)}, anchor=south east, legend columns=1, draw=none, fill=none},}
    \usepgfplotslibrary{fillbetween}

\usepackage{tikz}					            % sketches
    \usetikzlibrary{arrows.meta}
    \usetikzlibrary{calc}
    \usetikzlibrary{patterns}
    \usetikzlibrary{decorations.pathreplacing,calligraphy}

\usepackage[colorinlistoftodos]{todonotes}      % for ToDo notes
    \setuptodonotes{inline, fancyline} % color=blue!30,

\definecolor{LinkColorBlack}{RGB}{0,0,0}
\@ifundefined{\string\color@LinkColor}{\colorlet{LinkColor}{LinkColorBlack}}{}

\usepackage[%
    unicode = true,                     % loads with unicode support
    bookmarksopen = true,               % when starting with AcrobatReader, the Bookmarkcolumn is opened
    colorlinks=true,
    pdfpagelabels,
    bookmarksnumbered = true,
    plainpages = false,                 % correct, if pdflatex complains: ``destination with same identifier already exists''
    hypertexnames=true,
    pdfstartpage = 1,
    pdfstartview = FitH,                % (Fit, FitV, FitH) Anzeige der Startseite anpassen
    pdfremotestartview = FitH,
    pdftex,
    bookmarks=true,                     % generate bookmarks in PDF files
    pdfpagemode=UseNone,                % UseNone, UseOutlines, UseThumbs, FullScreen
    linkcolor = LinkColor,
    urlcolor = LinkColor,
    linkcolor = LinkColor,
    citecolor = LinkColor,
    anchorcolor = LinkColor,
    naturalnames = true                 % damit man Sonderzeichen in Refs nutzen kann
    ]{hyperref}
\numberwithin{equation}{section}    % das gibt an, wie equations gezählt werden (Eine Zahl Section und eine Zahl

\clearpairofpagestyles
\ohead{\normalfont \headmark} 		%Das hier steht oben außen
\automark{section}
\usepackage{algpseudocode}
\algblockdefx[Step]{Step}{End}[2][]{Step #1: #2}{\vspace*{-0.45cm}}
\algblockdefx[Input]{Input}{End}[1]{Input: #1}{\vspace*{-0.45cm}}
\algblockdefx[Output]{Output}{End}[1]{Output: #1}{\vspace*{-0.45cm}}
\makeatletter
\def\BState{\State\hskip-\ALG@thistlm}
\makeatother
\renewcommand{\b}[1]{\boldsymbol{\mathrm{#1}}} 		%für Fettgedrucktes, dass gerade geschrieben werden soll
\renewcommand{\bf}{\ensuremath{\b f}\xspace}

\newcommand{\bn}{\ensuremath{\b n}\xspace}

\newcommand{\bu}{\ensuremath{\b u}\xspace}

\newcommand{\bx}{\ensuremath{\b x}\xspace}

\newcommand{\R}{\ensuremath{\mathbb{R}}\xspace}

\newcommand{\N}{\ensuremath{\mathbb{N}}\xspace}

\newcommand{\cE}{\ensuremath{\mathcal{E}}\xspace}
\newcommand{\cF}{\ensuremath{\mathcal{F}}\xspace}

\newcommand{\cN}{\ensuremath{\mathcal{N}}\xspace}
\newcommand{\cO}{\ensuremath{\mathcal{O}}\xspace}

\newcommand{\cT}{\ensuremath{\mathcal{T}}\xspace}

\renewcommand{\tt}{\ensuremath{\tilde{t}}\xspace}

\newcommand{\tT}{\ensuremath{\tilde{T}}\xspace}

\newcommand{\dx}{\,\mathrm{d}\bx}
\newcommand{\ds}{\,\mathrm{d}s}

\newcommand{\phih}{\ensuremath{\phi_h}\xspace}
\newcommand{\tphih}{\ensuremath{\tphi_h}\xspace}

\newcommand{\psih}{\ensuremath{\psi_h}\xspace}

\newcommand{\tphi}{\ensuremath{\tilde{\phi}}\xspace}

\newcommandx{\NarrowBand}[2][1=\G, 2=]{\ensuremath{\Omega^{#2}_{#1}}\xspace}

\newcommandx{\2}[2][,2=\textbullet]{\begin{itemize}\item[{#2}] {#1}\end{itemize}}

\newcommand{\G}{\ensuremath{\Gamma}\xspace}
\newcommand{\Gh}{\ensuremath{\Gamma_h}\xspace}

\newcommand{\Gt}{\ensuremath{\Gamma(t)}\xspace}

\newcommandx{\BDF}[2][1=]{ {\widetilde{\partial}}^{#1}_t {#2} }         % BDF

\renewcommand{\emptyset}{\text{\usefont{OMS}{cmsy}{m}{n}\symbol{59}}}

\renewenvironment{math}[1][RCLRCLRCL]{\begin{IEEEeqnarray}{#1}}
{\end{IEEEeqnarray}\ignorespacesafterend}

\newenvironment{math*}[1][RCLRCLRCL]{\begin{IEEEeqnarray*}{#1}}
{\end{IEEEeqnarray*}\ignorespacesafterend}

\newenvironment{mathe*}[1][RCLRCLRCL]{\begin{IEEEeqnarray*}{#1}}
{\end{IEEEeqnarray*}\ignorespacesafterend}

\newenvironment{split2}[1][RCLRCLRCL]
{\begin{aligned}\begin{IEEEeqnarraybox}{#1}}
{\end{IEEEeqnarraybox}\end{aligned}\ignorespacesafterend}

\makeatletter
\patchcmd{\@IEEEeqnarray}{\relax}{\relax\intertext@}{}{}
\makeatother

\newtheorem{remark}{Remark}[section]
\newtheorem{assumption}{Assumption}[section]

%% file: LevelsetTransport_Revision_a.bbl
\begin{thebibliography}{10}

\bibitem{adalsteinsson1995fast}
{\sc D.~Adalsteinsson and J.~Sethian}, {\em A fast level set method for
  propagating interfaces}, Journal of computational physics, 118 (1995),
  pp.~269--277.

\bibitem{Adalsteinsson1999}
{\sc D.~Adalsteinsson and J.~Sethian}, {\em The fast construction of extension
  velocities in level set methods}, Journal of Computational Physics, 148
  (1999), pp.~2 -- 22.

\bibitem{AkrivisChenHanYuZhang2024}
{\sc G.~Akrivis, M.~Chen, J.~Han, F.~Yu, and Z.~Zhang}, {\em The variable
  two-step bdf method for parabolic equations}, BIT Numerical Mathematics, 64
  (2024).

\bibitem{Sussman2014}
{\sc M.~Arienti and M.~Sussman}, {\em An embedded level set method for
  sharp-interface multiphase simulations of diesel injectors}, Journal of
  Multiphase Flow, 59 (2014), pp.~1 -- 14.

\bibitem{basting2013minimization}
{\sc C.~Basting and D.~Kuzmin}, {\em A minimization-based finite element
  formulation for interface-preserving level set reinitialization}, Computing,
  95 (2013), pp.~13--25.

\bibitem{Feistauer2016}
{\sc E.~Bezchlebová, V.~Dolejší, and M.~Feistauer}, {\em Discontinuous
  {G}alerkin method for the solution of a transport level-set problem},
  Computers {\&} Mathematics with Applications, 72 (2016), pp.~455--480.

\bibitem{BrandnerReusken2020}
{\sc P.~Brandner and A.~Reusken}, {\em Finite element error analysis of surface
  {Stokes} equations in stream function formulation}, ESAIM: M2AN, 54 (2020),
  pp.~2069--2097.

\bibitem{Brezzi2004}
{\sc F.~Brezzi, L.~D. Marini, and E.~S\"{u}li}, {\em Discontinuous {G}alerkin
  methods for first-oder hyperbolic problems}, Mathematical Models and Methods
  in Applied Sciences, 14 (2004), pp.~1893--1903.

\bibitem{Burman2010a}
{\sc E.~Burman}, {\em Ghost penalty}, Comptes rendus mathematique, 348 (2010),
  pp.~1217--1220.

\bibitem{burman2015cutfem}
{\sc E.~Burman, S.~Claus, P.~Hansbo, M.~G. Larson, and A.~Massing}, {\em
  Cutfem: discretizing geometry and partial differential equations},
  International Journal for Numerical Methods in Engineering, 104 (2015),
  pp.~472--501.

\bibitem{BurmanHansbo2012}
{\sc E.~Burman and P.~Hansbo}, {\em Fictitious domain finite element methods
  using cut elements: {II}. a stabilized {N}itsche method}, Applied Numerical
  Mathematics, 62 (2012), pp.~328--341.

\bibitem{burman2010interior}
{\sc E.~Burman, A.~Quarteroni, and B.~Stamm}, {\em Interior penalty continuous
  and discontinuous finite element approximations of hyperbolic equations},
  Journal of Scientific Computing, 43 (2010), pp.~293--312.

\bibitem{ByrneHindmarsh1987}
{\sc G.~D. Byrne and A.~C. Hindmarsh}, {\em Stiff {ODE} solvers: {A} review of
  current and coming attractions}, Journal of Computational Physics, 70 (1987),
  pp.~1--62.

\bibitem{cho2011direct}
{\sc M.~H. Cho, H.~G. Choi, and J.~Y. Yoo}, {\em A direct reinitialization
  approach of level-set/splitting finite element method for simulating
  incompressible two-phase flows}, International Journal for Numerical Methods
  in Fluids, 67 (2011), pp.~1637--1654.

\bibitem{chopp1993computing}
{\sc D.~Chopp}, {\em Computing minimal surfaces via level set curvature flow},
  Journal of Computational Physics, 106 (1993), pp.~77--91.

\bibitem{Chopp2001}
{\sc D.~Chopp}, {\em Some improvements of the fast marching method}, SIAM J.
  Sci. Comput., 23 (2001), pp.~230--244.

\bibitem{CockburnHouShu1990}
{\sc B.~Cockburn, S.~Hou, and C.-W. Shu}, {\em The {Runge-Kutta} local
  projection discontinuous {{G}alerkin} finite element method for conservation
  laws. {IV}: The multidimensional case}, Mathematics of Computation, 54
  (1990), pp.~545--581.

\bibitem{CockburnShu1989}
{\sc B.~Cockburn and C.-W. Shu}, {\em {TVB} {Runge--Kutta} local projection
  discontinuous {{G}alerkin} finite element method for conservation laws {II}:
  General framework}, Mathematics of Computation, 52 (1989), pp.~411--435.

\bibitem{DiPietroErn2012}
{\sc D.~A. Di~Pietro and A.~Ern}, {\em Mathematical Aspects of Discontinuous
  {G}alerkin Methods}, Springer Berlin Heidelberg, 2012.

\bibitem{DiPietroLoForteParolini2006}
{\sc D.~A. {Di Pietro}, S.~{Lo Forte}, and N.~Parolini}, {\em Mass preserving
  finite element implementations of the level set method}, Applied Numerical
  Mathematics, 56 (2006), pp.~1179--1195.
\newblock Numerical Methods for Viscosity Solutions and Applications.

\bibitem{DolejsiFeistauer2015}
{\sc V.~Dolejší and M.~Feistauer}, {\em Discontinuous {G}alerkin Method:
  Analysis and Applications to Compressible Flow}, Springer International
  Publishing, 2015.

\bibitem{Dziuk1988}
{\sc G.~Dziuk}, {\em Finite elements for the {Beltrami} operator on arbitrary
  surfaces}, Hildebrandt, S., Leis, R., ``Partial Differential Equations and
  Calculus of Variations''. Lecture Notes in Mathematics, vol 1357. Springer,
  Berlin, Heidelberg,  (1988).

\bibitem{elliott2022numerical}
{\sc C.~M. Elliott, H.~Garcke, and B.~Kov{\'a}cs}, {\em Numerical analysis for
  the interaction of mean curvature flow and diffusion on closed surfaces},
  Numerische Mathematik, 151 (2022), pp.~873--925.

\bibitem{EnrightFedkiwFerzigerMitchell2002}
{\sc D.~Enright, R.~Fedkiw, J.~Ferziger, and I.~Mitchell}, {\em A hybrid
  particle level set method for improved interface capturing}, Journal of
  Computational Physics, 183 (2002), pp.~83--116.

\bibitem{FeistauerSvadlenka2004}
{\sc M.~Feistauer and K.~Švadlenka}, {\em Discontinuous {G}alerkin method of
  lines for solving nonstationary singularly perturbed linear problems},
  Journal of Numerical Mathematics, 12 (2004), pp.~97--117.

\bibitem{gomez2005reinitialization}
{\sc P.~G{\'o}mez, J.~Hern{\'a}ndez, and J.~L{\'o}pez}, {\em On the
  reinitialization procedure in a narrow-band locally refined level set method
  for interfacial flows}, International journal for numerical methods in
  engineering, 63 (2005), pp.~1478--1512.

\bibitem{GrandeLehrenfeldReusken2018}
{\sc J.~Grande, C.~Lehrenfeld, and A.~Reusken}, {\em Analysis of a high-order
  trace finite element method for pdes on level set surfaces}, SIAM Journal on
  Numerical Analysis, 56 (2018), pp.~228--255.

\bibitem{GrossReusken2011}
{\sc S.~Gross and A.~Reusken}, {\em Numerical Methods for Two-phase
  Incompressible Flows}, Springer Berlin Heidelberg, 2011.

\bibitem{Guermond2017}
{\sc J.-L. Guermond and A.~Ern}, {\em Finite element quasi-interpolation and
  best approximation}, ESAIM: Mathematical Modelling and Numerical Analysis, 51
  (2017), pp.~1367 -- 1385.

\bibitem{HairerWanner1996}
{\sc E.~Hairer and G.~Wanner}, {\em Solving Ordinary Differential Equations
  II}, Springer Berlin Heidelberg, 1996.

\bibitem{hu2007continuum}
{\sc D.~Hu, P.~Zhang, and E.~Weinan}, {\em Continuum theory of a moving
  membrane}, Physical Review E, 75 (2007), p.~041605.

\bibitem{jankuhn2018incompressible}
{\sc T.~Jankuhn, M.~Olshanskii, and A.~Reusken}, {\em Incompressible fluid
  problems on embedded surfaces: modeling and variational formulations},
  Interfaces and Free Boundaries, 20 (2018), pp.~353--377.

\bibitem{lee2014narrow}
{\sc C.~Lee, J.~Dolbow, and P.~J. Mucha}, {\em A narrow-band gradient-augmented
  level set method for multiphase incompressible flow}, Journal of
  Computational Physics, 273 (2014), pp.~12--37.

\bibitem{Lehrenfeld2017}
{\sc C.~Lehrenfeld}, {\em A higher order isoparametric fictitious domain method
  for level set domains}, in Geometrically Unfitted Finite Element Methods and
  Applications, S.~P.~A. Bordas, E.~Burman, M.~G. Larson, and M.~A. Olshanskii,
  eds., Springer International Publishing, 2017, pp.~65--92.

\bibitem{ngsxfem}
{\sc C.~Lehrenfeld, F.~Heimann, J.~Preuß, and H.~von Wahl}, {\em ngsxfem:
  Add-on to ngsolve for geometrically unfitted finite element discretizations}.
\newblock Journal of Open Source Software, 6(64), 3237,.

\bibitem{LehrenfeldOlshanskii2019}
{\sc C.~Lehrenfeld and M.~Olshanskii}, {\em An {Eulerian} finite element method
  for pdes in time-dependent domains}, ESAIM: M2AN, 53 (2019), pp.~585--614.

\bibitem{lehrenfeld2018stabilized}
{\sc C.~Lehrenfeld, M.~Olshanskii, and X.~Xu}, {\em A stabilized trace finite
  element method for partial differential equations on evolving surfaces}, SIAM
  Journal on Numerical Analysis, 56 (2018), pp.~1643--1672.

\bibitem{MarchandiseRemacleChevaugeon2006}
{\sc E.~Marchandise, J.-F. Remacle, and N.~Chevaugeon}, {\em A quadrature-free
  discontinuous {G}alerkin method for the level set equation}, Journal of
  Computational Physics, 212 (2006), pp.~338--357.

\bibitem{MassingLarsonLoggRognes2014}
{\sc A.~Massing, M.~G. Larson, A.~Logg, and M.~E. Rognes}, {\em A stabilized
  {N}itsche fictitious domain method for the {Stokes} problem}, Journal of
  Scientific Computing, 61 (2014), pp.~604--628.

\bibitem{ngsolve}
{\sc Netgen/Solve}.
\newblock {\url{https://ngsolve.org/}}.

\bibitem{ngo2017efficient}
{\sc L.~C. Ngo and H.~G. Choi}, {\em Efficient direct re-initialization
  approach of a level set method for unstructured meshes}, Computers \& Fluids,
  154 (2017), pp.~167--183.

\bibitem{nitschke2019hydrodynamic}
{\sc I.~Nitschke, S.~Reuther, and A.~Voigt}, {\em Hydrodynamic interactions in
  polar liquid crystals on evolving surfaces}, Physical Review Fluids, 4
  (2019), p.~044002.

\bibitem{olshanskii2023eulerian}
{\sc M.~Olshanskii, A.~Reusken, and P.~Schwering}, {\em An eulerian finite
  element method for tangential navier-stokes equations on evolving surfaces},
  Mathematics of Computation, 93 (2024), pp.~2031--2065.

\bibitem{osher2004level}
{\sc S.~Osher, R.~Fedkiw, and K.~Piechor}, {\em Level set methods and dynamic
  implicit surfaces}, Appl. Mech. Rev., 57 (2004), pp.~B15--B15.

\bibitem{Preuss2018}
{\sc J.~Preu{\ss}}, {\em Higher order unfitted isoparametric space-time FEM on
  moving domains}, mathesis, Institute for Numerical and Applied Mathematics
  University of G\"ottingen, 02 2018.

\bibitem{Rockswold1983}
{\sc G.~K. Rockswold}, {\em Stable variable step stiff methods for ordinary
  differential equations}, PhD thesis, Iowa State University, 1983.

\bibitem{RoosStynes}
{\sc H.-G. Roos, M.~Stynes, and L.~Tobiska}, {\em Robust Numerical Methods for
  Singularly Perturbed Differential Equations}, Springer Berlin Heidelberg,
  2008.

\bibitem{SchottWall2014}
{\sc B.~Schott and W.~A. Wall}, {\em A new face-oriented stabilized {XFEM}
  approach for {2D} and {3D} incompressible {Navier}--{Stokes} equations},
  Computer Methods in Applied Mechanics and Engineering, 276 (2014),
  pp.~233--265.

\bibitem{Sethian1996b}
{\sc J.~A. Sethian}, {\em A fast marching level set method for monotonically
  advancing fronts}, Proceedings of the National Academy of Sciences of the
  United States of America, 93 (1996), pp.~1591--1595.

\bibitem{Sethian1996}
\leavevmode\vrule height 2pt depth -1.6pt width 23pt, {\em Theory, algorithms,
  and applications of level set methods for propagating interfaces}, Acta
  Numerica, 5 (1996), pp.~309--395.

\bibitem{Sethian1999A}
\leavevmode\vrule height 2pt depth -1.6pt width 23pt, {\em Fast marching
  methods}, SIAM Review, 41 (1999), pp.~199--235.

\bibitem{SudirhamVanDerVegtVanDamme2006}
{\sc J.~Sudirham, J.~van~der Vegt, and R.~van Damme}, {\em Space-time
  discontinuous {G}alerkin method for advection-diffusion problems on
  time-dependent domains}, Applied Numerical Mathematics, 56 (2006),
  pp.~1491--1518.

\bibitem{Sussman1999}
{\sc M.~Sussman and E.~Fatemi}, {\em An efficient interface preserving level
  set redistancing algorithm and its application to interfacial incompressible
  fluid flow}, SIAM J. Sci. Comp., 20 (1999), pp.~1165 -- 1191.

\bibitem{Tornberg2000}
{\sc A.-K. Tornberg and B.~Engquist}, {\em A finite element based level-set
  method for multiphase flow applications}, Comp. Vis. Sci., 3 (2000),
  pp.~93--101.

\bibitem{torres2019modelling}
{\sc A.~Torres-S{\'a}nchez, D.~Mill{\'a}n, and M.~Arroyo}, {\em Modelling fluid
  deformable surfaces with an emphasis on biological interfaces}, Journal of
  fluid mechanics, 872 (2019), pp.~218--271.

\bibitem{SudirhamVanDerVegt2008}
{\sc J.~{van der Vegt} and J.~Sudirham}, {\em A space-time discontinuous
  {G}alerkin method for the time-dependent {Oseen} equations}, Applied
  Numerical Mathematics, 58 (2008), pp.~1892--1917.
\newblock Special Issue in Honor of Piet Hemker.

\bibitem{von2022unfitted}
{\sc H.~von Wahl, T.~Richter, and C.~Lehrenfeld}, {\em An unfitted {Eulerian}
  finite element method for the time-dependent {Stokes} problem on moving
  domains}, IMA Journal of Numerical Analysis, 42 (2022), pp.~2505--2544.

\bibitem{xue2021new}
{\sc T.~Xue, W.~Sun, S.~Adriaenssens, Y.~Wei, and C.~Liu}, {\em A new finite
  element level set reinitialization method based on the shifted boundary
  method}, Journal of Computational Physics, 438 (2021), p.~110360.

\end{thebibliography}
